\documentclass[12pt,a4paper]{article}
\usepackage[top=3cm, bottom=3cm, left=3cm, right=3cm]{geometry}
\usepackage[latin1]{inputenc}
\usepackage{amsmath}
\usepackage{amsthm}
\usepackage{amsfonts}
\usepackage{amssymb}
\usepackage{graphics}
\usepackage[hang,flushmargin]{footmisc} 

\usepackage{amsmath,amsthm,amssymb,graphicx, multicol, array}
\usepackage{enumerate}
\usepackage{enumitem}
\usepackage{xcolor}
\usepackage{currfile}

\usepackage[bookmarks,colorlinks,breaklinks, pagebackref]{hyperref}

\hypersetup{linkcolor=blue,citecolor=red,filecolor=blue,urlcolor=blue} 

\usepackage{tabto}

\usepackage{tikz}
\usepackage{pgfplots}

\DeclareMathOperator{\li}{li}

\DeclareMathOperator{\lcm}{lcm}

\newtheorem{thm}{Theorem}[section]

\newtheorem{lem}{Lemma}[section]

\newtheorem{conj}{Conjecture}[section]
\newtheorem{exa}{Example}[section]
\newtheorem{cor}{Corollary}[section]

\newtheorem{hyp}{Hypothesis}[section]
\newtheorem{prob}{Problem}[section]

\newcommand{\N}{\mathbb{N}}
\newcommand{\Z}{\mathbb{Z}}

\newcommand{\R}{\mathbb{R}}
\newcommand{\C}{\mathbb{C}}

\newcommand{\tP}{\mathbb{P}}

\usepackage{fancyhdr}
\title{Results for the Mobius Function and Liouville Function over the Shifted Primes}
\date{}
\author{N. A. Carella}

\begin{document}
\maketitle

\textbf{\textit{Abstract}:} 
This article provides new asymptotic conditional results for the summatory Mobius function $\sum_{p \leq x} \mu(p+a) =O \left (x(\log x)^{-c} \right )$ and the summatory Liouville function $\sum_{p \leq x} \lambda(p+a) =O \left (x(\log x)^{-c} \right )$ over the shifted primes, where $a\ne0$ is a fixed parameter, and $c>1$ is an arbitrary constant. These results improve the current estimates $\sum_{p \leq x} \mu(p+a)=(1-\delta)\pi(x)$, and $\sum_{p \leq x} \lambda(p+a)=(1-\delta)\pi(x)$ for $\delta>0$,  respectively. Furthermore, a conditional proof for the autocorrelation function $\sum_{p \leq x} \mu(p+a)\mu(p+b) =O \left (x(\log x)^{-c} \right )$, and an unconditional proof for the autocorrelation function $\sum_{p \leq x} \lambda(p+a)\lambda(p+b) =O \left (x(\log x)^{-c} \right )$ over the shifted primes, where $a\ne b$, are also included. \let\thefootnote\relax\footnote{\today \date{} 
	}

\tableofcontents
\newpage

\section{Introduction} \label{S1212P}
The Mobius function $\mu:\mathbb{N} \longrightarrow \{-1,0,1\}$ is defined by
\begin{equation}\label{eq1212P.050A}
	\mu(n) =
	\left \{
	\begin{array}{ll}
		(-1)^{v}     &n=p_1 p_2 \cdots p_w\\
		0           &n \ne p_1 p_2 \cdots p_w,\\
	\end{array}
	\right .
\end{equation}
where $n=p_1^{v_1} p_2^{v_2} \cdots p_w^{v_w}$, the $p_i\geq 2$ are primes, and $v_i\geq0$. Similarly, the Liouville function $\lambda:\mathbb{N} \longrightarrow \{-1,1\}$ is defined by
\begin{equation}\label{eq1212P.050B}
	\lambda(n) =
	(-1)^{p_1^{v_1} p_2^{v_2} \cdots p_w^{v_w}}    .
\end{equation}
The Mobius autocorrelation function 
\begin{equation} \label{eq1212P.060A}
	\sum_{n \leq x} \mu(n)\mu(n+a) 
\end{equation}and the Liouville autocorrelation function
\begin{equation} \label{eq1212P.060B}
	\sum_{n \leq x} \lambda(n)\lambda(n+a) 
\end{equation}
are topics of current research in several area of Mathematics, \cite{WD2016}, \cite{TJ2018},  \cite{MR2021}, et alii. Restricting the autocorrelation functions of multiplicative functions \eqref{eq1212P.060A} over the integers to the shifted primes reduce these functions to standard arithmetic averages over the shifted primes. For example, \eqref{eq1212P.060A} reduces to

\begin{equation} \label{eq1212P.070}
	\sum_{p \leq x} \mu(p)\mu(p+a)=-\sum_{p \leq x} \mu(p+a). 
\end{equation}	Similarly, an autocorrelation function of degree 3,
\begin{equation} \label{eq5980P.150}
	\sum_{n \leq x} \mu(n)\mu(n+a)\mu(n+b), 
\end{equation}
where $a,b\ne0$ such that $a\ne b$ are small fixed integers, over the integers reduces to an autocorrelation function of degree 2,
\begin{equation} \label{eq5980P.160}
	-\sum_{p \leq x} \mu(p+a)\mu(p+b)
\end{equation}		
over the shifted primes. Accordingly, these two open problems are equivalent. \\
		
Currently, the best asymptotic result for this summatory function is 
\begin{equation} \label{eq1212P.080}
	\sum_{p \leq x} \mu(p+a)=(1-\delta)\pi(x), 
\end{equation}
where $\delta>0$ is a constant, see \cite[Theorem 1]{HA1989}, and the restricted average
\begin{equation} \label{eq1212P.090}
	\sum_{a \leq z,}	\sum_{p \leq x} \mu(p+a)=o(z\pi(x)), 
\end{equation}
see \cite{LJ2021} for extensive details on recent developments on this topic. It is expected that $\sum_{p \leq x} \mu(p+a)=o(\pi(x))$. This note proposes the first nontrivial upper bound.

\begin{thm} \label{thm1212MP.200} Assume Hypothesis {\normalfont \ref{hyp2225P.550}}. Let $c>1$ be an arbitrary constant, and let $x>1$ be a large number. If $a \ne0$ is a small fixed integer, then
	\begin{equation} \label{eq1212T.200}
		\sum_{p \leq x} \mu(p+a) =O \left (\frac{x}{(\log x)^{c}} \right )\nonumber. 
	\end{equation}	
\end{thm}	
There is a similar result for the Liouville function, but it has a sharper error term. 
\begin{thm} \label{thm1212LP.200} Assume Hypothesis {\normalfont \ref{hyp2225P.550}}. Let $x>1$ be a large number. If $a \ne0$ is a small fixed integer, then
	\begin{equation} \label{eq1212T.210}
		\sum_{p \leq x} \lambda(p+a) =O \left (xe^{-c\sqrt{\log x}}\right )\nonumber, 
	\end{equation}	
where $c>0$ is an absolute constant.
\end{thm}

Furthermore, a conditional result for the Mobius autocorrelation function over the shifted primes achieves the followings asymptotic formula.
\begin{thm} \label{thm1212MP.700} Assume Hypothesis {\normalfont \ref{hyp2225P.560}}. Let $x>1$ be a large number, and let $a,b \in \Z$ be small fixed integers such that $a\ne b$. Then,
	\begin{equation} \label{eq1212P.700}
		\sum_{p \leq x} \mu(p+a)\mu(p+b) =O \left (\frac{x}{(\log x)^{c}} \right )\nonumber, 
	\end{equation}	
where $c>1$ is an arbitrary constant. 
\end{thm}

Unexpectedly, the matching result for the Liouville autocorrelation function is unconditional and has a sharper error term.

\begin{thm} \label{thm1212LP.700} Assume Hypothesis {\normalfont \ref{hyp2225P.560}}. Let $x>1$ be a large number, and let $a,b \in \Z$ be small fixed integers such that $a\ne b$. Then,
	\begin{equation} \label{eq1212P.710}
		\sum_{p \leq x} \lambda(p+a)\lambda(p+b) =O \left (xe^{-c\sqrt{\log x}}\right )\nonumber, 
	\end{equation}	
where $c>0$ is an absolute constant.
\end{thm}

The essential foundational materials for the analysis of the Mobius function are covered in Section \ref{S2222MP}. The proof of Theorem \ref{thm1212MP.200} for the summatory Mobius function over the shifted primes is assembled in Subsection \ref{S8111MP}, and the proof for the Mobius autocorrelation function over the shifted primes is assembled in Subsection \ref{S2626MP}. 

The essential foundational materials for the analysis of the Liouville function are covered in Section \ref{S2299}.
The proof of Theorem \ref{thm1212LP.200} for the summatory Liouville function over the shifted primes is assembled in Subsection \ref{S1616LP}. The proof of Theorem \ref{thm1212LP.700} for the Liouville autocorrelation function over the shifted primes appears in Subsection \ref{S2626LP}\\

For completeness some results on the Mobius autocorrelation function and the Liouville autocorrelation function over the integers, based on the sign patterns technique, are included in Section \ref{S7979MN} and Section \ref{S7979LN} respectively. 


\section{Results for the Mobius Function over the Shifted Primes}\label{S2222MP}

\subsection{Average Orders of Mobius Functions}\label{S2222P}
\begin{thm} \label{thm2222.500} If $\mu: \N\longrightarrow \{-1,0,1\}$ is the Mobius function, then, for any large number $x>1$, the following statements are true.
	\begin{enumerate} [font=\normalfont, label=(\roman*)]
		\item $\displaystyle \sum_{n \leq x} \mu(n)=O \left (xe^{-c\sqrt{\log x}}\right )$, \tabto{8cm} unconditionally,
		\item $\displaystyle \sum_{n\leq x}\mu(n)=O\left( x^{1/2+\varepsilon} \right ), $ \tabto{8cm} conditional on the RH,
	\end{enumerate}where $c>0$ is an absolute constant, and $\varepsilon>0$ is an arbitrarily small number.
\end{thm}
\begin{proof}[\textbf{Proof}]  See \cite[p.\ 6]{DL2012}, \cite[p.\ 182]{MV2007}, \cite[p.\ 347]{HW2008}, et alii.   
\end{proof}

There are many sharp bounds of the summatory function of the Mobius function, say, $O(xe^{-c(\log x)^{\delta}})$, and the conditional estimate $O(x^{1/2+\varepsilon})$ presupposes that the nontrivial zeros of the zeta function $ \zeta(\rho)=0$ in the critical strip $\{0<\Re e(s)<1 \}$ are of the form $\rho=1/2+it, t \in \mathbb{R}$. However, the simpler notation will be used whenever it is convenient.

\begin{thm} \label{thmMF222.090} Let $x\geq1$ be a large number, and let $q\ll(\log x)^B$, where $B\geq0$ is an arbitrary constant. If $1\leq a<q$ are relatively prime integers, then, 
	\begin{equation}
		\sum_{\substack{n \leq x\\n\equiv a \bmod q}} \mu(n)=O \left (\frac{x}{\log^{C}x}\right )\nonumber, 
	\end{equation}
	where $C=C(B)>0$ is a constant. 
\end{thm}
\begin{proof}[\textbf{Proof}]A sketch of the proof appears in \cite[p.\ 385]{MV2007}.
\end{proof}

\subsection{Twisted Exponential Sums}
One of the earliest result for twisted sums is stated below.
\begin{thm} \label{thm3970M.300} {\normalfont (\cite{DH1937})} If $\alpha$ is a real number, and $D>0$ is an arbitrary constant, then
	\begin{equation}\label{eq3970M.310} 
		\sup_{\alpha\in\R}\sum_{n \leq x} \mu(n)e^{i 2 \pi  \alpha n}<\frac{c_1x}{(\log x)^{D}} \nonumber,
	\end{equation}
	where $c_1=c_1(D)>0$ is a constant depending on $D$, as the number $x \to \infty$.
\end{thm}

\subsection{Mean Values and  Hypotheses over the Shifted Primes }\label{S2225P}
There are several mean values and equidistribution results for arithmetic functions over arithmetic progressions of level of distribution $\theta<1/2$. The best known case is the Bombieri-Vinogradov theorem, see \cite[Theorem 15.4]{DL2012}, the case for the Mobius function is proved in \cite[Theorem 1]{WD1973} and \cite{SW1971} states the following.

\begin{cor} {\normalfont (\cite[Corollary 1]{SW1971} }\label{cor2225P.550}  Let $a\geq 1$ be a fixed parameter, and let $x \geq 1$ be a large number. If $C>0$ is a constant, then
	\begin{equation}
		\sum_{q\leq x^{1/2}/\log^B x} \max_{a \bmod q}\max_{z \leq x}  \bigg |\sum_{\substack{n \leq z \\
				n \equiv a \bmod q}} \mu(n) \bigg |\ll \frac{x}{(\log x)^{C}},
	\end{equation}where the constant $B>0$ depends on $C$.
\end{cor}

However, there is no literature on the mean values and equidistribution for arithmetic functions over arithmetic progressions of shifted prime. A comparable result is expected to hold.

\begin{hyp} \label{hyp2225P.550}  Let $a\geq 1$ be a fixed parameter, and let $x \geq 1$ be a large number. If $C>0$ is a constant, then
\begin{enumerate} [font=\normalfont, label=(\roman*)]
\item $\displaystyle \sum_{q\leq x^{1/2}/\log^B x} \max_{a \bmod q}\max_{z \leq x}  \bigg |\sum_{\substack{p \leq z \\
		p \equiv d \bmod q}} \lambda(p+a) \bigg |\ll \frac{x}{(\log x)^{C}},$
\item $\displaystyle \sum_{q\leq x^{1/2}/\log^B x} \max_{a \bmod q}\max_{z \leq x}  \bigg |\sum_{\substack{p \leq z \\
		p \equiv d \bmod q}} \mu(p+a) \bigg |\ll \frac{x}{(\log x)^{C}}, $ 
\end{enumerate}
where the constant $B>0$ depends on $C$.
\end{hyp}
 
A similar mean value for the autocorrelation over arithmetic progression is of interest in the theory of arithmetic correlation functions.
\begin{hyp} \label{hyp2225P.560}  Let $a,b\in \Z$ be a fixed pair of small integers such that $a\ne b$, and let $x \geq 1$ be a large number. If $C>0$ is a constant, then
\begin{enumerate} [font=\normalfont, label=(\roman*)]
	\item $\displaystyle \sum_{q\leq x^{1/2}/\log^B x} \max_{d \bmod q}\max_{z \leq x}  \bigg |\sum_{\substack{p \leq z \\
			p \equiv d \bmod q}} \lambda(p+a) \lambda(p+b)\bigg |\ll \frac{x}{(\log x)^{C}},$
	\item $\displaystyle \sum_{q\leq x^{1/2}/\log^B x} \max_{d \bmod q}\max_{z \leq x}  \bigg |\sum_{\substack{p \leq z \\
			p \equiv d \bmod q}} \mu(p+a) \mu(p+b)\bigg |\ll \frac{x}{(\log x)^{C}}, $ 
\end{enumerate}
where the constant $B>0$ depends on $C$.
\end{hyp}

\subsection{Mobius Function over the Shifted Primes}\label{S8111MP}
The analysis of the average order over the shifted primes
\begin{equation}\label{eq1616MP.200A}
	R(a,x)=	\sum_{p \leq x} \mu(p+a)
\end{equation}
is currently viewed as an intractable problem, the restricted double average order 
\begin{equation}\label{eq1616MP.210A}
	\sum_{a \leq z}\sum_{p \leq x} \mu(p+a)=o\left( z\pi(x)\right) 
\end{equation}
is the only result available in the literature, see \cite[Theorem 1.1]{LJ2021} for the exact details. An asymptotic formula for the Mobius function over the shifted primes is considered in this section. Theorem \ref{thm1212MP.200} is the same as the result proved below.

\begin{thm}\label{thm8111MP.500}  Let $x\geq 1$ be a large number, and let $\mu: \mathbb{N} \longrightarrow \{-1,0,1\}$ be the Mobius function. If $a\ne0$ is a fixed integer, then, 
	\begin{equation}\label{eq8111MP.500}
		\sum_{p \leq x}\mu(p+a)=O\left( \frac{x}{(\log x)^c}\right) ,\nonumber
	\end{equation}
	where $c> 0$ is an arbitrary constant.
\end{thm}

\begin{proof} Let $x_0=x^{1/2-\varepsilon}$, with $\varepsilon>0$. Substitute the identity $\mu(n)\lambda(n)\mu^2(n)$, and partition the equivalent finite sum.
\begin{eqnarray}\label{eq8111MP.520}
	\sum_{p \leq x}\mu(p+a)&=&\sum_{p \leq x}\lambda(p+a)\mu^2(p+a)\\
	&=&\sum_{p \leq x}\lambda(p+a)\sum_{d^2\mid p+a}\mu(d)\nonumber\\
	&=&\sum_{d^2\leq x}\mu(d)\sum_{\substack{p \leq x\\d^2\mid p+a}}\lambda(p+a)\nonumber\\
	&=&\sum_{d^2\leq x_0}\mu(d)\sum_{\substack{p \leq x\\d^2\mid p+a}}\lambda(p+a)+\sum_{x_0<d^2\leq x}\mu(d)\sum_{\substack{p \leq x\\d^2\mid p+a}}\lambda(p+a)\nonumber.
\end{eqnarray}	
By Hypothesis \ref{hyp2225P.550}, the first term has the upper bound
\begin{eqnarray}\label{eq8111MP.530}
\sum_{d^2\leq x_0}\mu(d)\sum_{\substack{p \leq x\\d^2\mid p+a}}\lambda(p+a)&\leq&\sum_{d^2\leq x_0}\bigg | \sum_{\substack{p \leq x\\p\equiv -a\bmod d^2}}\lambda(p+a)\bigg |\\
	&=&O\left(\frac{x}{(\log x)^{c}}\right) \nonumber,
\end{eqnarray}	
where $c>1$ is an arbitrary constant. The second term has the upper bound
\begin{eqnarray}\label{eq8111MP.540}
\sum_{x_0<d^2\leq x}\mu(d)\sum_{\substack{p \leq x\\d^2\mid p+a}}\lambda(p+a)&\leq&\sum_{x_0<d^2\leq x}\sum_{\substack{p \leq x\\p\equiv -a\bmod d^2}}1\\
	&=&O\left( x^{1/2+\varepsilon}\right) \nonumber.
\end{eqnarray}	
The sum of \eqref{eq8111MP.530} and \eqref{eq8111MP.540} completes the verification.
\end{proof}

\subsection{Squarefree Shifted Primes}\label{S8009MP}
Let $p\geq2$ be  prime, and let $a\ne0$ be a fixed integer. The number of squarefree integers of the form $p+a$, or totient integers, have the following asymptotic formula expressed in terms of the logarithm integral $\li(x)=\int_2^x(\log t)^{-1}dt$.

\begin{thm}\label{thm8009MP.350A}  Let $x\geq 1$ be a large number, and let $\mu: \mathbb{N} \longrightarrow \{-1,0,1\}$ be the Mobius function. If $a\ne0$ is a fixed integer, then, 
	\begin{equation}\label{eq8009MP.350A}
		\sum_{p \leq x}\mu(p+a)^2=s_0\li(x)+O\left( \frac{x}{(\log x)^c}\right) ,\nonumber
	\end{equation}
	where $s_0= 0.373955\ldots $ is a constant, and $c> 1$ is an arbitrary constant.
\end{thm}
\begin{proof}[\textbf{Proof}] Substituting the identity $\mu(n)^2=\sum_{d^2\mid n}\mu(d)$, and switching the order of summation yield
	\begin{eqnarray}\label{eq8009MP.360A}
		\sum_{p \leq x}\mu(p+a)^2 
		&=&\sum_{p \leq x} \sum_{d^2\mid p+a}\mu(d)\\
		&=&\sum_{d^2 \leq x}\mu(d) \sum_{\substack{p\leq x\\d^2\mid p+a}}1\nonumber\\
		&=&\sum_{d^2 \leq x_0}\mu(d) 
		\sum_{\substack{p\leq x\\d^2\mid p+a}}1+
		\sum_{x_0< d^2 \leq x}\mu(d) 
		\sum_{\substack{p\leq x\\d^2\mid p+a}}1\nonumber,
	\end{eqnarray}
	where $x_0=(\log x)^{2c}$, and $c>1$ is a constant. Applying the Siegel-Walfisz theorem, see \cite[p.\ 405]{FI2010}, \cite[Theorem 15.3]{DL2012}, et cetera, to the first subsum in the partition yields
	\begin{eqnarray}\label{eq8009MP.370A}
		\sum_{d^2 \leq x_0}\mu(d) 
		\sum_{\substack{p\leq x\\d^2\mid p+a}}1&=&\sum_{d^2 \leq x_0}\mu(d) \left(\frac{\li(x)}{\varphi(d^2)}+ O\left(xe^{-c_0\sqrt{\log x}}\right)\right)\\
		&=&\li(x)\sum_{n \geq 1} \frac{\mu(n)}{\varphi(n^2)} +O\left( \frac{x}{(\log x)^c}\right) \nonumber,
	\end{eqnarray}
	where $c_0>0$ is an absolute constant. An estimate of the second subsum in the partition yields
	\begin{eqnarray}\label{eq8009MP.380A}
		\sum_{x_0< d^2 \leq x}\mu(d) 
		\sum_{\substack{p\leq x\\d^2\mid p+a}}1&\leq&\sum_{x_0< d^2 \leq x,} \sum_{\substack{p\leq x\\d^2\mid p+a}}1\\
		&\ll&x\sum_{x_0<d^2 \leq x}\frac{1}{d^2}\nonumber\\
		&\ll&\frac{x}{(\log x)^c}\nonumber.
	\end{eqnarray}
	Summing \eqref{eq8009MP.370A} and \eqref{eq8009MP.380A} completes the verification.
\end{proof}
The well known constant has the numerical approximation
\begin{equation}\label{eq8009MP.390A}
s_0=\sum_{n \geq 1} \frac{\mu(n)}{\varphi(n^2)}=\prod_{p\geq 2}\left ( 1-\frac{1}{p(p-1)}\right )=0.373955838964330040631201
 \ldots.
\end{equation}

\subsection{Nonlinear Autocorrelation Functions Results}\label{S8008PP}
The asymptotic formula for squarefree autocorrelation function over the shifted primes is evaluated now.
\begin{thm} \label{thm8008MP.500} Let $\mu: \mathbb{N} \longrightarrow \{-1,0,1\}$ be the Mobius function, and let $a,b\ne0$ be integers such that $a\ne b$. Then, for any sufficiently large number $x\geq1$, 
	\begin{equation} \label{eq8008MP.500}
		\sum_{p \leq x} \mu^2(p+a) \mu^2(p+b)=s_0^2\li(x)   +O\left(  \frac{x}{(\log x)^c}\right) \nonumber,
	\end{equation}
where $s_0=0.373955\ldots$, and $c>1$ is an arbitrary constant. 
\end{thm}
\begin{proof}[\textbf{Proof}] Substitute the identity $\mu(n)^2=\sum_{d^2\mid n}\mu(d)$, and switching the order of summation yield
	\begin{eqnarray}\label{eq8008MP.510}
		\sum_{p \leq x}\mu^2(p+a) \mu^2(p+b)&=&\sum_{p \leq x} \sum_{d^2\mid p+a}\mu(d)\sum_{e^2\mid p+b}\mu(e)\\
		&=&\sum_{d^2 \leq x}\mu(d) \sum_{e^2 \leq x}\mu(e)\sum_{\substack{p\leq x\\d^2\mid p+a,\;e^2\mid p+b }}1\nonumber\\
		&=&\sum_{d^2 \leq x_0}\mu(d) \sum_{e^2 \leq x_0}\mu(e)\sum_{\substack{p\leq x\\d^2\mid p+a,\;e^2\mid p+b }}1\nonumber\\
		&& \hskip .75 in +\sum_{x_0<d^2 \leq x}\mu(d) \sum_{x_0<e^2 \leq x}\mu(e)\sum_{\substack{p\leq x\\d^2\mid p+a,\;e^2\mid p+b }}1,\nonumber
	\end{eqnarray}
	where $x_0=(\log x)^c$, and $c>1$ is an arbitrary constant. Let $q=d^2e^2$, $\gcd(d,e)=1$, and $p\equiv f \bmod q$. Applying the Siegel-Walfisz theorem, see \cite[p.\ 405]{FI2010}, \cite[Theorem 15.3]{DL2012}, et cetera, to the first subsum in the partition yields
	\begin{eqnarray}\label{eq8008MP.520}
T(x)&=&	\sum_{d^2 \leq x_0}\mu(d) \sum_{e^2 \leq x_0}\mu(e)\sum_{\substack{p\leq x\\d^2\mid p+a,\;e^2\mid p+b }}1\\	&=&\sum_{d^2 \leq x_0}\mu(d) \sum_{e^2 \leq x_0}\mu(e)\left(\frac{\li(x)}{\varphi(d^2 e^2)}+ O\left(xe^{-c\sqrt{\log x}}\right)\right)\nonumber\\
&=&\sum_{d^2 \leq x_0}\frac{\mu(d)}{\varphi(d^2)} \sum_{e^2 \leq x_0}\frac{\mu(e)}{\varphi(e^2)}\left(\li(x)+ O\left(xe^{-c_0\sqrt{\log x}}\right)\right)\nonumber\\	
	&=&\li(x) \left(\sum_{n \geq 1} \frac{\mu(n)}{\varphi(n^2)}\right)^2  +O\left(  \frac{x}{(\log x)^c}\right)   \nonumber,
	\end{eqnarray}
	where the constant $s_0=\sum_{n \geq 1} \mu(n)\varphi(n^2)^{-1}$ is computed in \eqref{eq8009MP.390A}, and $c_0>0$ is an absolute constant. An estimate of the second subsum in the partition yields
	\begin{eqnarray}\label{eq8008MP.530}
		\sum_{x_0<d^2 \leq x}\mu(d) \sum_{x_0<e^2 \leq x}\mu(e)\sum_{\substack{p\leq x\\d^2\mid p+a,\;e^2\mid p+b }}1&\leq&\sum_{x_0<d^2 \leq x} \sum_{x_0<e^2 \leq x}\sum_{\substack{p\leq x\\d^2\mid p+a,\;e^2\mid p+b }}1\nonumber\\
		&\ll&x\sum_{x_0<d^2 \leq x}\frac{1}{d^2}\sum_{x_0<e^2 \leq x}\frac{1}{e^2}\\
		&\ll&\frac{x}{(\log x)^c}\nonumber.
	\end{eqnarray}
	Summing \eqref{eq8008MP.520} and \eqref{eq8008MP.530} completes the verification.
\end{proof}

\begin{thm}\label{thm8008MP.700} Assume Hypothesis {\normalfont \ref{hyp2225P.550}}. Let $x\geq 1$ be a large number, and let $\mu: \mathbb{N} \longrightarrow \{-1,0,1\}$ be the Mobius function. If $a,b\ne0$ are small fixed integers such that $a\ne b$, then, for any sufficiently large number $x\geq1$, 
	\begin{equation}\label{eq8008MP.710}
		\sum_{p \leq x}\mu(p+a)^2 \mu(p+b)=O\left (\frac{x}{(\log x)^c} \right ),\nonumber
	\end{equation}
	where $c> 0$ is an arbitrary constant.
\end{thm}
\begin{proof}[\textbf{Proof}] Substitute the identity $\mu(n)^2=\sum_{d^2\mid n}\mu(d)$, and switching the order of summation yield
	\begin{eqnarray}\label{eq8009MP.360}
		\sum_{p \leq x}\mu(p+a)^2 \mu(p+b)&=&\sum_{p \leq x} \mu(p+b)\sum_{d^2\mid p+a}\mu(d)\\
		&=&\sum_{d^2 \leq x}\mu(d) \sum_{\substack{p\leq x\\d^2\mid p+a}}\mu(p+b)\nonumber\\
		&=&\sum_{d^2 \leq x^{2\varepsilon}}\mu(d) \sum_{\substack{p\leq x\\d^2\mid p+a}}\mu(p+b)+\sum_{x^{2\varepsilon}< d^2 \leq x}\mu(d) \sum_{\substack{p\leq x\\d^2\mid p+a}}\mu(p+b)\nonumber,
	\end{eqnarray}
	where $\varepsilon\in (0,1/4)$. Applying Hypothesis \ref{hyp2225P.550} to the first subsum in the partition yields
	\begin{eqnarray}\label{eq8009MP.370}
		\sum_{d^2 \leq x^{2\varepsilon}}\mu(d) \sum_{\substack{n\leq x\\d^2\mid n}}\mu(n+t)&\leq&\sum_{q \leq x^{\varepsilon}} \bigg | \sum_{\substack{p\leq x\\p\equiv -a \bmod q}}\mu(p+b)\bigg |\\
		&=&O\left( \frac{x}{(\log x)^{c}}\right) \nonumber,
	\end{eqnarray}
	where $q=d^2$. An estimate of the second subsum in the partition yields
	\begin{eqnarray}\label{eq8009MP.380}
		\sum_{x^{2\varepsilon}< d^2 \leq x}\mu(d) \sum_{\substack{p\leq x\\d^2\mid p+a}}\mu(p+b)&\leq&\sum_{x^{2\varepsilon}< d^2 \leq x} \sum_{\substack{p\leq x\\d^2\mid p+a}}1\\
		&\ll&x\sum_{x^{2\varepsilon}< d^2 \leq x} \frac{1}{d^2}\nonumber\\
		&\ll&x^{1-\varepsilon}\nonumber.
	\end{eqnarray}
	Summing \eqref{eq8009MP.370} and \eqref{eq8009MP.380} completes the verification.
\end{proof}

\begin{thm}\label{thm8008MP.750} Assume Hypothesis {\normalfont \ref{hyp2225P.560}}. Let $x\geq 1$ be a large number, and let $\mu: \mathbb{N} \longrightarrow \{-1,0,1\}$ be the Mobius function. If $a_0,a_1,a_2\ne0$ are small fixed integers such that $a_0<a_1<a_2$, then, for any sufficiently large number $x\geq1$, 
	\begin{equation}\label{eq8008MP.750}
		\sum_{p \leq x}\mu(p+a_0)^2 \mu(p+a_1)\mu(p+a_2)=O\left (\frac{x}{(\log x)^c} \right ),\nonumber
	\end{equation}
	where $c> 0$ is an arbitrary constant.
\end{thm}
\begin{proof}[\textbf{Proof}] Except for minor changes, the verification is similar to the previous proof, for example, repeat \eqref{eq8009MP.360} to \eqref{eq8009MP.380}, but uses Hypothesis \ref{hyp2225P.560}.
\end{proof}

\subsection{Single-Sign Pattern Mobius Characteristic Functions} \label{S8833MP}
The analysis of single-Sign pattern characteristic function 
\begin{equation}\label{eq8833MP.100}
\mu^{\pm}(n)=\mu^2(n)\left( \frac{1\pm\mu(n)}{2}\right)= 	\begin{cases}
1 &\text{ if } \mu(n)=\pm1,\\
0 &\text{ if } \mu(n)\ne\pm1,
\end{cases}
\end{equation}
of the subset of integers
\begin{equation}\label{eq8833MP.110}
\mathcal{P}_{\mu}^{\pm}	=	\{n\geq 1: \mu(n)=\pm\}
\end{equation} 
is well known. Here, the same idea is extended to the shifted primes.

\begin{lem}\label{lem8833MP.200A} Let $a\ne0$ be an integer, and let $\mu(n)\in \{-1,0,1\}$ be the Mobius function. Then,
\begin{eqnarray}\label{eq8833MP.200A}
	\mu^{\pm}(a,p)&=&\mu^2(p+a)\left( \frac{1\pm\mu(p+a)}{2}\right)\\	&=&
\begin{cases}
		1 &\text{ if } \mu(p+a)=\pm1,\\
		0 &\text{ if } \mu(p+a)\ne\pm1,\nonumber\\
\end{cases}
\end{eqnarray}
are the characteristic functions of the subset of primes
\begin{equation}\label{eq8833MP.210A}
	\mathcal{P}_{\mu}^{\pm }(a)	=	\{p\geq 2: \mu(p+a)=\pm1\}.
\end{equation} 
\end{lem}

\subsection{Single-Sign Pattern Mobius Counting Functions} \label{S8844MP} 
The single-sign patterns $\mu(n)=1$ and $\mu(n)=-1$ counting functions over the integers have the forms
\begin{equation}\label{eq8844MP.100A}
	R^{+}(x)=\sum_{\substack{n\leq x\\ \mu(n)=1}}1=\sum_{n\leq x}\left( \frac{1+\mu(n)}{2}\right) \mu^2(n)=\frac{1}{2}\frac{x}{\zeta(2)}+O\left( xe^{-c\sqrt{\log x}}\right), 
\end{equation}
and 
\begin{equation}\label{eq8844MP.100B}
	R^{-}(x)=\sum_{\substack{n\leq x\\ \mu(n)=-1}}1=\sum_{n\leq x}\left( \frac{1-\mu(n)}{2}\right) \mu^2(n)=\frac{1}{2}\frac{x}{\zeta(2)}+O\left( xe^{-c\sqrt{\log x}}\right), 
\end{equation}
respectively. In terms of these functions, the summatory Mobius function has the asymptotic formula
\begin{equation}\label{eq8844MP.100C}
	R(x)=\sum_{n\leq x}\mu(n)=R^{+}(x)-R^{-}(x)=O\left( xe^{-c\sqrt{\log x}}\right). 
\end{equation}
Basically, it is a different form of the Prime Number Theorem
\begin{equation}\label{eq8844MP.100D}
\pi(x)=\li(x)+O\left(xe^{-c\sqrt{\log x}} \right) ,	
\end{equation}
where $\li(x)=\int_2^2(\log t)^{-1}dt$ is the logarithm integral, and $c>0$ is an absolute constant, see \cite[Eq.~27.12.5]{DLMF}, \cite[Theorem 3.10]{EL1985}, et alii. \\

The same idea extends to the single-sign patterns $\mu(p+a)=1$ and $\mu(p+a)=-1$ of the shifted primes. Let $R(a,x)=\sum_{p\leq x}\mu(p+a)$. The single-sign pattern counting functions are defined by
\begin{equation}\label{eq8844MP.110A}
	R^{+}(a,x)=\sum_{\substack{p\leq x\\ \mu(p+a)=1}}1=\sum_{\substack{p\leq x\\ n\in \mathcal{P}_{\mu}^{+}(a)}}1,
\end{equation}
\begin{equation}\label{eq8844MP.110B}
	R^{-}(a,x)=\sum_{\substack{p\leq x\\ \mu(p+a)=-1}}1=\sum_{\substack{p\leq x\\ p\in \mathcal{P}_{\mu}^{-}(a)}}1.
\end{equation}

The single-sign pattern counting functions \eqref{eq8844MP.110A} to \eqref{eq8844MP.110B} are precisely the cardinalities of the subsets of integers
\begin{multicols}{2}
	\begin{enumerate}
		\item $\mathcal{P}_{\mu}^{+}(a)\subset \tP$ ,
		\item $\mathcal{P}_{\mu}^{-}(a)\subset \tP$ ,
	\end{enumerate}
\end{multicols}
defined in \eqref{eq8833MP.210A}. The symbol $\tP =\{2,3,5,\ldots, \}$ denotes the set of prime numbers. \\

The next result is required to complete the analysis of the asymptotic formula for $R(a,x)=\sum_{p\leq x}\mu(p+a)$, which is completed in the next section.

\begin{lem}\label{lem8844MP.300A} Let $x\geq1$ be a large number, and let $a\ne0$ be a fixed integer. Then, 
	\begin{equation}\label{eq8844MP.300A}
		R^{\pm}(a,x)= \frac{1}{2}s_0\li(x)+ \frac{1}{2}R(a,x)+O\left(\frac{x}{(\log x)^{c}} \right) ,\nonumber
	\end{equation}
	where $s_0=0.373955\ldots$, and $c>0$ are constants.
\end{lem}

\begin{proof}[\textbf{Proof}] Consider the pattern $\mu(p+a)=+1$. Now, use Lemma \ref{lem8833MP.200A} to express the single-sign pattern counting function as 
	\begin{eqnarray}   \label{eq4422P.310A}
		2R^{+}(a,x)&=&\sum_{p\leq x}\mu^{+}(a,p)\\
		&=&\sum_{p\leq x}\mu^2(p+a)\left( 1+\mu(p+a)\right)\nonumber\\
		&=&\sum_{p\leq x}\mu^2(p+a)+\sum_{p\leq x}\mu^3(p+a)\nonumber.
	\end{eqnarray}
	The two finite sums have the following evaluations or estimates.
	\begin{enumerate}
		\item $ \displaystyle \sum_{p\leq x}\mu^2(p+a)=s_0\li(x)+O\left(\frac{x}{(\log x)^{c}} \right), $\\
		
		where $s_0=0.373955\ldots $ is a constant, see Theorem \ref{thm8009MP.350A}.\\
				
		\item $ \displaystyle \sum_{n\leq x}\mu^3(p+a)=\sum_{p\leq x}\mu(p+a), $\\
		
		since $\mu^{2k+1}(n)=\mu(n)$ for any integer $k$. 
	\end{enumerate}
	Summing these evaluations or estimates verifies the claim for $R^{+}(a,x)\geq0$. The verification for the other single-sign pattern counting function $R^{-}(a,x)\geq0$, is similar.
\end{proof}

\subsection{Double-Sign Patterns Mobius Characteristic Functions} \label{S8855MP}
The analysis of single-sign pattern characteristic functions is extended here to the double-sign patterns
\begin{equation}\label{eq8855MP.100}
	\mu(p+a)=\pm1 \quad \text{ and }\quad 	\mu(p+b)=\pm1, 
\end{equation}
where $a,b\ne0$ such that $a\ne b$, and $p\geq2$ is prime. The current research has emphasized the sign patterns over the integers, confer \cite{HA1986}, \cite[Corollary 1.7]{TT2015}, \cite{KP1986}, and there is some literature on the patterns over the shifted primes, \cite{PJ2015}, and similar literature, for details.

\begin{lem}\label{lem8855MP.200} Let $a,b\ne0$ such that $a\ne b$ be small fixed integers, and let $\mu(n)\in \{-1,0,1\}$ be the Mobius function. Then,
	\begin{eqnarray}\label{eq8855MP.200}
		\mu^{\pm \pm}(a,b,p)&=&\mu^2(p+a)\mu^2(p+b)\left( \frac{1\pm\mu(p+ a)}{2}\right)\left( \frac{1\pm\mu(p+b)}{2}\right)\\	&=&
		\begin{cases}
			1 &\text{ if } \mu(p+a)=\pm1,\mu(p+b)=\pm1,\\
			0 &\text{ if } \mu(p+a)\ne\pm1,\mu(p+b)\ne\pm1,\nonumber\\
		\end{cases}
	\end{eqnarray}
	are the characteristic functions of the subset of integers
	\begin{equation}\label{eq8855MP.210}
		\mathcal{P}_{\mu}^{\pm \pm}(a,b)	=	\{p\geq 2: \mu(p+a)=\pm1, \mu(p+b)=\pm1\}.
	\end{equation} 
\end{lem}

\subsection{Double-Sign Patterns Mobius Counting Functions} \label{S8877MP}
The single-sign counting functions for the single patterns $\mu(p+a)=1$ and $\mu(p+a)=-1$ are extended to the double-sign counting functions for the double signs patterns 
\begin{equation}\label{eq8877MP.100}
	(\mu(p+a),\mu(p+b))=(\pm1,\pm1).
\end{equation}

The double-sign counting functions are defined by
\begin{equation}\label{eq8877MP.110A}
	R^{++}(a,b,x)=\sum_{\substack{p\leq x\\ \mu(p+a)=1,\; \mu(p+b)=1}}1=\sum_{\substack{p\leq x\\ p\in \mathcal{P}_{\mu}^{++}(a,b)}}1,
\end{equation}
\begin{equation}\label{eq8877MP.110B}
	R^{+-}(a,b,x)=\sum_{\substack{p\leq x\\ \mu(p+a)=1,\; \mu(p+b)=-1}}1=\sum_{\substack{p\leq x\\ p\in \mathcal{P}_{\mu}^{+-}(a,b)}}1,
\end{equation}
\begin{equation}\label{eq8877MP.110C}
	R^{-+}(a,b,x)=\sum_{\substack{p\leq x\\ \mu(p+a)=-1,\; \mu(p+b)=1}}1=\sum_{\substack{p\leq x\\ p\in \mathcal{P}_{\mu}^{-+}(a,b)}}1,
\end{equation}
\begin{equation}\label{eq8877MP.110D}
	R^{- -}(a,b,x)=\sum_{\substack{p\leq x\\ \mu(p+a)=-1,\; \mu(p+b)=-1}}1=\sum_{\substack{p\leq x\\ p\in \mathcal{P}_{\mu}^{- -}(a,b)}}1.
\end{equation}

The double-sign counting functions \eqref{eq8877MP.110A} to \eqref{eq8877MP.110D} are precisely the cardinalities of the subsets of integers
\begin{multicols}{2}
	\begin{enumerate}
		\item $\mathcal{P}_{\mu}^{++}(a,b)\subset \tP$ ,
		\item $\mathcal{P}_{\mu}^{+-}(a,b)\subset \tP$ ,
		\item $\mathcal{P}_{\mu}^{-+}(a,b)\subset \tP$ ,
		\item $\mathcal{P}_{\mu}^{--}(a,b)\subset \tP$ ,
	\end{enumerate}
\end{multicols}
defined in \eqref{eq8855MP.210}. In terms of these functions, the double autocorrelation function has form 
\begin{eqnarray}\label{eq8877MP.120}
	R(a,b,x)&=&\sum_{p\leq x}\mu(p+a)\mu(p+b)\\
	&=&R^{++}(a,b,x)-R^{+-}(a,b,x)+R^{--}(a,b,x)-R^{-+}(a,b,x).\nonumber
\end{eqnarray}
The next result is required to complete the analysis of the asymptotic formula for $R(a,b,x)$, which is completed in the next section.

\begin{lem}\label{lem8877MP.300} Assume Hypothesis {\normalfont \ref{hyp2225P.550}}. Let $x\geq1$ be a large number, and let $a,b\in \Z$ such that $a\ne b$, be fixed integers. Then, 
	\begin{equation}\label{eq8877MP.300}
		R^{\pm \pm}(a,b,x)= \frac{1}{4}s_0^2\li(x)+ \frac{1}{4}R(a,b,x)+O\left(\frac{x}{(\log x)^{c}} \right) ,\nonumber
	\end{equation}
	where $s_0=0.373955\ldots$, and $c>1$ is an arbitrary constants.
\end{lem}

\begin{proof}[\textbf{Proof}] Without loss in generality, consider the pattern $(\mu(p+a),\mu(p+b))=(+1,+1)$. Now, use Lemma \ref{lem8855MP.200} to express the double-sign pattern counting function as 
	\begin{eqnarray}   \label{eq4477MP.310}
		4R^{++}(a,b,x)&=&\sum_{p\leq x}\mu^{++}(a,b,p)\\
		&=&\sum_{p\leq x}\mu^2(p+a)\mu^2(p+b)\left( 1+\mu(p+a)\right) \left( 1+\mu(p+b)\right)\nonumber\\
		&=&\sum_{p\leq x}\mu^2(p+a)\mu^2(p+b)\left( 1+\mu(p+a)+\mu(p+b)+\mu(p+a)\mu(p+b)\right)\nonumber\\
		&=&\sum_{p\leq x}\mu^2(p+a)\mu^2(p+b)+\sum_{p\leq x}\mu^3(p+a)\mu^2(p+b) \nonumber\\
		&&\hskip 1 in +\sum_{p\leq x}\mu(p+a)^2\mu^3(p+b)+\sum_{p\leq x}\mu^3(p+a)\mu^3(p+b)\nonumber\\
		&\geq&0\nonumber.
	\end{eqnarray}
	The last four finite sums have the following evaluations or estimates.
	\begin{enumerate}
		\item $ \displaystyle \sum_{p\leq x}\mu^2(p+a)\mu^2(p+b)=s_0^2\li(x)+O\left(\frac{x}{(\log x)^{c}} \right), $\\
		
		where $s_0>0$ is a constant, see Theorem \ref{thm8008MP.500}.\\
		
		\item $ \displaystyle \sum_{p\leq x}\mu^3(p+a)\mu^2(p+b)=O\left(\frac{x}{(\log x)^{c}} \right), $\\
		
		where $c>0$ is an arbitrary constant, see Theorem \ref{thm8008MP.700}.\\
		
		\item $ \displaystyle \sum_{p\leq x}\mu^2(p+a)\mu^3(p+b)=O\left(\frac{x}{(\log x)^{c}} \right), $\\
		
		where $c>0$ is an arbitrary constant, see Theorem \ref{thm8008MP.700}.\\
		
		\item $ \displaystyle \sum_{p\leq x}\mu^3(p+a)\mu^3(p+b)=\sum_{p\leq x}\mu(p+a)\mu(p+b), $\\
		
		since $\mu^{2k+1}(n)=\mu(n)$ for any integer $k$. 
	\end{enumerate}
	Summing these evaluations or estimates verifies the claim for $R^{++}(a,b,x)\geq0$. The verifications for the next three double-sign pattern counting functions $R^{+-}(a,b,x)\geq0$, $R^{-+}(a,b,x)\geq0$, and $R^{--}(a,b,x)\geq0$ are similar.
\end{proof}

\subsection{Result for the Mobius Autocorrelation Function} \label{S2626MP}

\begin{proof}[\textbf{Proof}] ({\bfseries Theorem \ref{thm1212MP.700}}) Utilize the identity $\mu(n)=\lambda(n)\mu^2(n)$ to rewrite the autocorrelation function.  
	\begin{eqnarray}\label{eq2626MP.220A}
		R(a,b,x)&=&\sum_{p\leq x}\mu(p+a)\mu(p+b)\\
		&=&\sum_{p\leq x}\lambda(p+a)\mu^2(p+a) \lambda(p+b)\mu^2(p+b) \nonumber\\
		&=&\sum_{p\leq x}\mu(p+a)\mu(p+b)\\
		&=&\sum_{d^2\leq x}\mu(d)\sum_{e^2\leq x}\mu(e)\sum_{\substack{p\leq x\\d^2\mid p+a\\e^2\mid p+b}}\lambda(p+a) \lambda(p+b)  \nonumber. 
	\end{eqnarray} 
Let $x_0\leq x^{1/8-\varepsilon}$, where $\varepsilon>0$, and partition it.\\

Applying Hypothesis \ref{hyp2225P.550} to the first term yields the upper bound
\begin{eqnarray}\label{eq2626MP.230A}
T_0(a,b,x)&=&\sum_{d^2\leq x_0}\mu(d)\sum_{e^2\leq x_0}\mu(e)\sum_{\substack{p\leq x\\d^2\mid p+a\\e^2\mid p+b}}\lambda(p+a) \lambda(p+b)\\
&\leq&\sum_{q\leq x_0^4}\bigg |\sum_{\substack{p\leq x\\p\equiv f\bmod q}}\lambda(p+a) \lambda(p+b) \bigg | \nonumber\\ 
&=&O\left(\frac{x}{(\log x)^c} \right) ,
\end{eqnarray}
where $q=\lcm(d^2,e^2)\leq x_0^4\leq x^{1/2-\varepsilon}$, and $c>1$ is an arbitrary constant.\\

The second term has the upper bound
\begin{eqnarray}\label{eq2626MP.240A}
	T_1(a,b,x)&=&\sum_{d^2\leq x_0}\mu(d)\sum_{x_0<e^2\leq x}\mu(e)\sum_{\substack{p\leq x\\d^2\mid p+a\\e^2\mid p+b}}\lambda(p+a) \lambda(p+b)\\
	&\leq&\sum_{x_0<d^2\leq x,}\sum_{x_0<e^2\leq x}\sum_{\substack{p\leq x\\d^2\mid p+a\\e^2\mid p+b}}1  \nonumber\\
&=&O\left( x^{1-\varepsilon}\right) \nonumber. 
\end{eqnarray}
The third term has the upper bound
\begin{eqnarray}\label{eq2626MP.250A}
	T_2(a,b,x)&=&\sum_{x_0<d^2\leq x}\mu(d)\sum_{e^2\leq x_0}\mu(e)\sum_{\substack{p\leq x\\d^2\mid p+a\\e^2\mid p+b}}\lambda(p+a) \lambda(p+b)\\
	&\leq&\sum_{x_0<d^2\leq x,}\sum_{e^2\leq x_0}\sum_{\substack{p\leq x\\d^2\mid p+a\\e^2\mid p+b}}1  \nonumber\\
	&=&O\left( x^{1-\varepsilon}\right)\nonumber.
\end{eqnarray} 
The fourth term has the upper bound
\begin{eqnarray}\label{eq2626MP.260A}
	T_3(a,b,x)&=&\sum_{x_0<d^2\leq x}\mu(d)\sum_{x_0<e^2\leq x}\mu(e)\sum_{\substack{p\leq x\\d^2\mid p+a\\e^2\mid p+b}}\lambda(p+a) \lambda(p+b)\\
	&\leq&\sum_{x_0<d^2\leq x,}\sum_{x_0<e^2\leq x_0}\sum_{\substack{p\leq x\\d^2\mid p+a\\e^2\mid p+b}}1  \nonumber\\
	&=&O\left( x^{1-\varepsilon}\right)\nonumber.
\end{eqnarray}  
Summing \eqref{eq2626MP.230A} to \eqref{eq2626MP.260A} completes the proof.
\end{proof}

\section{Results for the Liouville Function over the Shifted Primes}\label{S2299}
\subsection{Average Orders of Liouville Functions}\label{S2299B}

\begin{thm} \label{thm2299.500} If $\lambda: \N\longrightarrow \{-1,1\}$ is the Liouville function, then, for any large number $x>1$, the following statements are true.
\begin{enumerate} [font=\normalfont, label=(\roman*)]
	\item $\displaystyle \sum_{n \leq x} \lambda(n)=O \left (xe^{-c\sqrt{\log x}}\right )$, \tabto{8cm} unconditionally,
	\item $\displaystyle \sum_{n\leq x}\lambda(n)=O\left( x^{1/2+\varepsilon} \right ), $ \tabto{8cm}conditional on the RH.
\end{enumerate}
where $c>0$ is an absolute constant, and $\varepsilon>0$ is an arbitrarily small number.
\end{thm}

\begin{proof} [\textbf{Proof}] (i) Substitute the identity $\lambda(n)=\sum_{d^2\mid n}\mu(n/d^2)$ to convert the summatory function to 
	\begin{eqnarray}\label{eq2299L.300}
		\sum_{n\leq x}\lambda(n)
		&=&\sum_{n\leq x}\sum_{d^2\mid n}\mu(n/d^2)\\
		&=&\sum_{d^2\leq x}M(x/d^2)\nonumber,
	\end{eqnarray} 
	where $M(t)=\sum_{n\leq t}\mu(n)$. Applying Theorem \ref{thm2222.500} yields
	\begin{eqnarray}\label{eq2299L.310}
		\sum_{d^2\leq x}M(x/d^2)
		&=&O\left(xe^{-c\sqrt{\log x}}\sum_{d\leq \sqrt{x}}\frac{1}{d^2}\right) \\
		&=&O\left(xe^{-c\sqrt{\log x}} \right) \nonumber,  
	\end{eqnarray} 
	where $c>0$ is an absolute constant.	
\end{proof}

\subsection{Single-Sign Patterns Liouville Characteristic Functions} \label{S8833LP}
The analysis of single-sign pattern characteristic function 
\begin{equation}\label{eq8833LP.100}
	\lambda^{\pm}(n)=\left( \frac{1\pm\lambda(n)}{2}\right)= 	\begin{cases}
		1 &\text{ if } \lambda(n)=\pm1,\\
		0 &\text{ if } \lambda(n)\ne\pm1,
	\end{cases}
\end{equation}
of the subset of integers
\begin{equation}\label{eq8833LP.110}
	\mathcal{N}_{\lambda}^{\pm}	=	\{n\geq 1: \lambda(n)=\pm\}
\end{equation} 
is well known. Here, the same idea is extended to the shifted primes.

\begin{lem}\label{lem8833LP.200A} Let $a\ne0$ be an integer, and let $\lambda(n)\in \{-1,1\}$ be the Liouville function. Then,
	\begin{eqnarray}\label{eq8833LP.200A}
		\lambda^{\pm}(a,p)&=&\left( \frac{1\pm\lambda(p+a)}{2}\right)\\	&=&
		\begin{cases}
			1 &\text{ if } \lambda(p+a)=\pm1,\\
			0 &\text{ if } \lambda(p+a)\ne\pm1,\nonumber\\
		\end{cases}
	\end{eqnarray}
	are the characteristic functions of the subset of primes
	\begin{equation}\label{eq8833LP.210A}
		\mathcal{P}_{\lambda}^{\pm }(a)	=	\{p\geq 2: \lambda(p+a)=\pm1\}.
	\end{equation} 
\end{lem}

\subsection{Single-Sign Pattern Liouville Counting Functions} \label{S8844LP}
The single-sign patterns $\lambda(n)=1$ and $\lambda(n)=-1$ single-sign pattern counting functions over the integers have the forms
\begin{equation}\label{eq8844LP.100A}
	Q^{+}(x)=\sum_{\substack{n\leq x\\ \lambda(n)=1}}1=\sum_{n\leq x}\left( \frac{1+\lambda(n)}{2}\right) =\frac{1}{2}[x]+O\left( xe^{-c\sqrt{\log x}}\right), 
\end{equation}
and 
\begin{equation}\label{eq8844LP.100B}
	Q^{-}(x)=\sum_{\substack{n\leq x\\ \lambda(n)=-1}}1=\sum_{n\leq x}\left( \frac{1-\lambda(n)}{2}\right) =\frac{1}{2}[x]+O\left(xe^{-c\sqrt{\log x}}\right), 
\end{equation}
where $[x]$ is the largest integer function, respectively. In terms of the single-sign pattern counting functions, the summatory function has the asymptotic formula
\begin{equation}\label{eq8844LP.100C}
	Q(x)=\sum_{n\leq x}\lambda(n)=Q^{+}(x)-Q^{-}(x)=O\left( xe^{-c\sqrt{\log x}}\right). 
\end{equation}
Basically, it is a form of the prime number theorem, stated in \eqref{eq8844MP.100D}. \\

The same idea extends to the single-sign patterns $\lambda(p+a)=1$ and $\lambda(p+a)=-1$ of the shifted primes. Let $Q(a,x)=\sum_{p\leq x}\lambda(p+a)$. The single-sign patterns counting functions are defined by
\begin{equation}\label{eq8844LP.110A}
	Q^{+}(a,x)=\sum_{\substack{p\leq x\\ \lambda(p+a)=1}}1=\sum_{\substack{p\leq x\\ n\in \mathcal{B}^{+}(a)}}1,
\end{equation}
\begin{equation}\label{eq8844LP.110B}
	Q^{-}(a,x)=\sum_{\substack{p\leq x\\ \lambda(p+a)=-1}}1=\sum_{\substack{p\leq x\\ p\in \mathcal{B}^{-}(a)}}1.
\end{equation}

The single-sign patterns counting functions \eqref{eq8844LP.110A} to \eqref{eq8844LP.110B} are precisely the cardinalities of the subsets of integers
\begin{multicols}{2}
	\begin{enumerate}
		\item $\mathcal{P}_{\lambda}^{+}(a)\subset \tP$ ,
		\item $\mathcal{P}_{\lambda}^{-}(a)\subset \tP$ ,
	\end{enumerate}
\end{multicols}
defined in \eqref{eq8833LP.210A}. The symbol $\tP =\{2,3,5,\ldots, \}$ denotes the set of prime numbers. \\

The next result is required to complete the analysis of the asymptotic formula for $Q(a,x)=\sum_{p\leq x}\lambda(p+a)$, which is completed in the next section.

\begin{lem}\label{lem8844LP.300A} Let $x\geq1$ be a large number, and let $a\ne0$ be a fixed integer. Then, 
	\begin{equation}\label{eq8844LP.300A}
		Q^{\pm}(a,x)= \frac{1}{2}\li(x)+\frac{1}{2}Q(a,x)+O\left(xe^{-c\sqrt{\log x}} \right) ,\nonumber
	\end{equation}
where $c>0$ ia an absolute constant.
\end{lem}

\begin{proof}[\textbf{Proof}] Consider the pattern $\lambda(p+a)=+1$. Now, use Lemma \ref{lem8833LP.200A} to express the single-sign pattern counting function as 
	\begin{eqnarray}   \label{eq442LP.310A}
		2Q^{+}(a,x)&=&\sum_{p\leq x}\lambda^{+}(a,p)\\
		&=&\sum_{p\leq x}\left( 1+\lambda(p+a)\right)\nonumber\\
		&=&\sum_{p\leq x}1+\sum_{p\leq x}\lambda(p+a)\nonumber.
	\end{eqnarray}
The first finite sum is $\sum_{p\leq x}1=\pi(x)$, see \eqref{eq8844MP.100D} for more details, and the second sum is $Q(a,x)/2$.
Summing these evaluations or estimates verifies the claim for $Q^{+}(a,x)\geq0$. The verification for the next single-sign pattern counting function $Q^{-}(a,x)\geq0$, is similar.
\end{proof}

\subsection{Result for Liouville Summatory Function} \label{S1616LP}
The analysis of the average order over the shifted primes

\begin{equation}\label{eq1616LP.200A}
	Q(a,x)=	\sum_{p \leq x} \lambda(p+a)
\end{equation}
is currently viewed as an intractable problem, the restricted double average order 
\begin{equation}\label{eq1616LP.210A}
	\sum_{a \leq z}\sum_{p \leq x} \lambda(p+a)=o\left( z\pi(x)\right) 
\end{equation}

is the only result available in the literature, see \cite{PJ2015}, and \cite[Theorem 1.1]{LJ2021} for the exact details. However, the introduction of the single-sign counting function $Q^{+}(a,x)$ and $Q^{-}(a,x)$ transforms the problem into an elementary problem. 

\begin{proof} [\textbf{Proof}]({\bfseries Theorem \ref{thm1212LP.200}}) Let $x_0=x^{1/2-\varepsilon}$, with $\varepsilon>0$. Utilize the identity $\lambda(n)=\sum_{d^2\mid n}\mu(n/d^2)$ to rewriting the Liouville summatory function in terms of the Mobius function, and partition the finite sum.
\begin{eqnarray}\label{eq1616LP.220A}
Q(a,x)&=&\sum_{p\leq x}\lambda(p+a)\\
&=&\sum_{p\leq x}\sum_{d^2\mid p+a}\mu((p+a)/d^2) \nonumber\\
&=&\sum_{d^2\leq x} \sum_{\substack{p\leq x\\d^2\mid p+a}}\mu((p+a)/d^2)\nonumber\\
&=&\sum_{d^2\leq x_0} \sum_{\substack{p\leq x\\d^2\mid p+a}}\mu((p+a)/d^2)+\sum_{x_0<d^2\leq x} \sum_{\substack{p\leq x\\d^2\mid p+a}}\mu((p+a)/d^2)\nonumber. 
\end{eqnarray} 
By Hypothesis \ref{hyp2225P.550}, the first term has the upper bound
\begin{eqnarray}\label{eq1616LP.230}
\sum_{d^2\leq x_0} \sum_{\substack{p\leq x\\d^2\mid p+a}}\mu((p+a)/d^2)&\leq&\sum_{d^2\leq x_0} \bigg | \sum_{\substack{p\leq x\\d^2\mid p+a}}\mu((p+a)/d^2) \bigg |\\
	&=&\left( \frac{x}{(\log x)^c}\right) \nonumber,
\end{eqnarray} 
where $c>0$ is an arbitrary constant. The second term has the upper bound
\begin{eqnarray}\label{eq1616LP.240}
\sum_{x_0<d^2\leq x} \sum_{\substack{p\leq x\\d^2\mid p+a}}\mu((p+a)/d^2)&\leq&\sum_{x_0<d^2\leq x} \sum_{\substack{p\leq x\\d^2\mid p+a}}1\\
&\leq&x\sum_{x_0<d^2\leq x}\frac{1}{d^2} \nonumber\\
	&=&O\left( x^{1/2+\varepsilon}\right)\nonumber. 
\end{eqnarray} 
The sum of \eqref{eq1616LP.230} and \eqref{eq1616LP.240} completes the verification.

\end{proof}

\subsection{Result for the Liouville-Mobius Correlation Function} \label{S1616LM}

\begin{thm} \label{thm1616LM.900} Let $c>1$ be an arbitrary constant, and let $x>1$ be a large number. If $a \ne0$ is a small fixed integer, then
	\begin{equation} \label{eq1616LM.900}
		\sum_{n \leq x} \mu(n)\lambda(n+a) =O \left (\frac{x}{(\log x)^{c}} \right )\nonumber. 
	\end{equation}	
\end{thm}	

\begin{proof} [\textbf{Proof}] Let $x_0=(\log x)^{2c}$, with $c>0$ is arbitrary. Rewriting the summatory Mobius-Liouville correlation function in terms of the Mobius function, and partition the finite sum.
	\begin{eqnarray}\label{eq1616LM.220A}
D(a,x)&=&\sum_{n\leq x}\mu(n)\lambda(n+a)\\
		&=&\sum_{n\leq x}\mu(n)\sum_{d^2\mid n+a}\mu((n+a)/d^2) \nonumber\\
		&=&\sum_{d^2\leq x}\mu((n+a)/d^2) \sum_{\substack{n\leq x\\d^2\mid n+a}}\mu(n)\nonumber\\
		&=&\sum_{d^2\leq x_0}\mu((n+a)/d^2) \sum_{\substack{n\leq x\\d^2\mid n+a}}\mu(n)+\sum_{x_0<d^2\leq x}\mu((n+a)/d^2) \sum_{\substack{n\leq x\\d^2\mid n+a}}\mu(n)\nonumber. 
	\end{eqnarray} 
	Applying Theorem \ref{thmMF222.090} to the first term yields the asymptotic formula
	\begin{eqnarray}\label{eq1616LM.230}
\sum_{d^2\leq x_0}\mu((n+a)/d^2) \sum_{\substack{n\leq x\\d^2\mid n+a}}\mu(n)&=&\sum_{d^2\leq x_0}\mu((n+a)/d^2) \cdot O\left(\frac{x}{d^2}e^{-c_0\sqrt{\log x}} \right) \nonumber\\
		&=&O\left(xe^{-c_0\sqrt{\log x}} \sum_{d^2\leq x}\frac{1}{d^2}\right) \\
		&=&O\left(xe^{-c_0\sqrt{\log x}} \right)\nonumber, 
	\end{eqnarray} 
where $c_0>0$ is an absolute constant. The second term has the upper bound	
	\begin{eqnarray}\label{eq1616LM.240}
	\sum_{x_0<d^2\leq x}\mu((n+a)/d^2) \sum_{\substack{n\leq x\\d^2\mid n+a}}\mu(n)&\leq&\sum_{x_0<d^2\leq x,} \sum_{\substack{n\leq x\\d^2\mid n+a}}1\\
		&\leq&x\sum_{x_0<d^2\leq x} \frac{1}{d^2}\nonumber\\
		&=&O\left(\frac{x}{(\log x)^c} \right)    \nonumber,  
	\end{eqnarray} 
	where $c>0$ is an arbitrary constant.
\end{proof}

\subsection{Double-Sign Patterns Liouville Characteristic Functions} \label{S8855LP}
The analysis of single-sign pattern characteristic functions is extended here to the double-sign patterns
\begin{equation}\label{eq8855LP.100}
	\lambda(p+a)=\pm1 \quad \text{ and }\quad 	\lambda(p+b)=\pm1, 
\end{equation}
where $a,b\ne0$ such that $a\ne b$, and $p\geq2$ is prime. The current research has emphasized the sign patterns over the integers, confer \cite{HA1986}, \cite[Corollary 1.7]{TT2015}, \cite{KP1986}, \cite{SA2022}, and there is some literature on the sign patterns over the shifted primes, \cite{PJ2015}, and similar literature, for details.

\begin{lem}\label{lem8855LP.200} Let $a,b\in \Z$ such that $a\ne b$ be small fixed integers, and let $\lambda(n)\in \{-1,1\}$ be the Liouville function. Then,
	\begin{eqnarray}\label{eq8855LP.200}
		\lambda^{\pm \pm}(a,b,p)&=&\left( \frac{1\pm \lambda(p+ a)}{2}\right)\left( \frac{1\pm\lambda(p+b)}{2}\right)\\	&=&
		\begin{cases}
			1 &\text{ if } \lambda(p+a)=\pm1,\mu(p+b)=\pm1,\\
			0 &\text{ if } \lambda(p+a)\ne\pm1,\mu(p+b)\ne\pm1,\nonumber\\
		\end{cases}
	\end{eqnarray}
	are the characteristic functions of the subset of integers
	\begin{equation}\label{eq8855LP.210}
		\mathcal{P_{\lambda}}^{\pm \pm}(a,b)	=	\{p\geq 2: \lambda(p+a)=\pm1, \lambda(p+b)=\pm1\}.
	\end{equation} 
\end{lem}

\subsection{Double-Sign Patterns in Liouville Counting Functions} \label{S8877LP}
The counting functions for the single-sign patterns $\lambda(p+a)=1$ and $\lambda(p+a)=-1$ are extended to the counting functions for the double-sign patterns 
\begin{equation}\label{eq8877LP.100}
	(\lambda(p+a),\lambda(p+b))=(\pm1,\pm1).
\end{equation}

The double-sign patterns counting functions are defined by
\begin{equation}\label{eq8877LP.110A}
	Q^{++}(a,b,x)=\sum_{\substack{p\leq x\\ \lambda(p+a)=1,\; \lambda(p+b)=1}}1=\sum_{\substack{p\leq x\\ p\in \mathcal{P_{\lambda}}^{++}(a,b)}}1,
\end{equation}
\begin{equation}\label{eq8877LP.110B}
	Q^{+-}(a,b,x)=\sum_{\substack{p\leq x\\ \lambda(p+a)=1,\; \lambda(p+b)=-1}}1=\sum_{\substack{p\leq x\\ p\in \mathcal{P_{\lambda}}^{+-}(a,b)}}1,
\end{equation}
\begin{equation}\label{eq8877LP.110C}
	Q^{-+}(a,b,x)=\sum_{\substack{p\leq x\\ \lambda(p+a)=-1,\; \lambda(p+b)=1}}1=\sum_{\substack{p\leq x\\ p\in \mathcal{P_{\lambda}}^{-+}(a,b)}}1,
\end{equation}
\begin{equation}\label{eq8877LP.110D}
	Q^{- -}(a,b,x)=\sum_{\substack{p\leq x\\ \lambda(p+a)=-1,\; \lambda(p+b)=-1}}1=\sum_{\substack{p\leq x\\ p\in \mathcal{P_{\lambda}}^{- -}(a,b)}}1.
\end{equation}

The double-sign patterns counting functions \eqref{eq8877LP.110A} to \eqref{eq8877LP.110D} are precisely the cardinalities of the subsets of integers
\begin{multicols}{2}
	\begin{enumerate}
		\item $\mathcal{P_{\lambda}}^{++}(a,b)\subset \tP$ ,
		\item $\mathcal{P_{\lambda}}^{+-}(a,b)\subset \tP$ ,
		\item $\mathcal{P_{\lambda}}^{-+}(a,b)\subset \tP$ ,
		\item $\mathcal{P_{\lambda}}^{--}(a,b)\subset \tP$ ,
	\end{enumerate}
\end{multicols}
defined in \eqref{eq8855LP.210}. The symbol $\tP =\{2,3,5,\ldots, \}$ denotes the set of prime numbers. In terms of the double-sign patterns counting functions, the autocorrelation function has form 
\begin{eqnarray}\label{eq8877L.120}
	Q(a,b,x)&=&\sum_{p\leq x}\lambda(p+a)\lambda(p+b)\\
	&=&Q^{++}(a,b,x)-Q^{+-}(a,b,x)+Q^{--}(a,b,x)-Q^{-+}(a,b,x).\nonumber
\end{eqnarray}
The next result is required to complete the analysis of the asymptotic formula for $Q(a,b,x)$, which is completed in the next section.

\begin{lem}\label{lem8877LP.300} Let $x\geq1$ be a large number, and let $a,b\in \Z$ such that $a\ne b$, be fixed integers. If $\lambda: \mathbb{Z} \longrightarrow \{-1,1\}$ is the Liouville function, then, 
	\begin{equation}\label{eq8877LP.300}
		Q^{\pm \pm}(a,b,x)= \frac{1}{4}\li(x)+ \frac{1}{4}Q(a,b,x)+O\left(xe^{-c\sqrt{\log x}} \right) ,\nonumber
	\end{equation}
	where $c>0$ is an absolute constant.
\end{lem}

\begin{proof}[\textbf{Proof}] Without loss in generality, consider the pattern $(\lambda(p+a),\lambda(p+b))=(+1,+1)$. Now, use Lemma \ref{lem8855LP.200} to express the double-sign pattern counting function as 
	\begin{eqnarray}   \label{eq8877LP.310}
		4Q^{++}(a,b,x)&=&\sum_{p\leq x}\lambda^{++}(a,b,p)\\
		&=&\sum_{p\leq x}\left( 1+\lambda(p+a)\right) \left( 1+\lambda(p+b)\right)\nonumber\\
		&=&\sum_{p\leq x}\left( 1+\lambda(p+a)+\lambda(p+b)+\lambda(p+a)\lambda(p+b)\right)\nonumber\\
		&=&\sum_{p\leq x}1+\sum_{p\leq x}\lambda(p+a)  +\sum_{p\leq x}\lambda(p+b)+\sum_{p\leq x}\lambda(p+a)\lambda(p+b)\nonumber\\
		&\geq&0\nonumber.
	\end{eqnarray}
	The first three finite sums on the last line have the following evaluations or estimates.
	\begin{enumerate}
		\item $ \displaystyle \sum_{p\leq x}1=\li(x)+O\left(xe^{-c\sqrt{\log x}} \right) , $ \tabto{8cm}confer \eqref{eq8844MP.100D} for more details,\\
		\item $ \displaystyle \sum_{p\leq x}\lambda(p+a)=O \left (xe^{-c\sqrt{\log x}}\right )$, \tabto{8cm}see Theorem \ref{thm2299.500},\\
		\item $ \displaystyle \sum_{p\leq x}\lambda(p+b)=O \left (xe^{-c\sqrt{\log x}}\right )$, \tabto{8cm}see Theorem \ref{thm2299.500}.
	\end{enumerate}
In all cases $c>0$ is an absolute constant, unconditionally. Summing these evaluations or estimates verifies the claim for $Q^{++}(a,b,x)\geq0$. The verifications for the next three double-sign pattern counting functions $Q^{+-}(a,b,x)\geq0$, $Q^{-+}(a,b,x)\geq0$, and $Q^{--}(a,b,x)\geq0$ are similar.
\end{proof}

\subsection{Result for the Liouville Autocorrelation Function} \label{S2626LP}
Unlike the conditional result for the Mobius autocorrelation function in Theorem \ref{thm1212LP.700}, the matching result for the Liouville autocorrelation function is unconditional, and has a sharper error term.

\begin{proof}[\textbf{Proof}] ({\bfseries Theorem \ref{thm1212LP.700}}) Utilize the identity $\lambda(n)=\sum_{d^2\mid n}\mu(n/d^2)$ to rewriting the Liouville summatory function in terms of the Mobius function.
	\begin{eqnarray}\label{eq2626LP.220A}
		Q(a,b,x)&=&\sum_{p\leq x}\lambda(p+a)\lambda(p+b)\\
		&=&\sum_{p\leq x} \sum_{d^2\mid p+a}\mu((p+a)/d^2)\sum_{e^2\mid p+b}\mu((p+b)/e^2)\nonumber\\
&=& \sum_{d^2\leq x} \sum_{e^2\leq x} \sum_{\substack{p\leq x\\ d^2\mid p+a\\e^2\mid p+b}}\mu((p+a)/d^2)\mu((p+b)/e^2) \nonumber.
	\end{eqnarray} 
Let $x_0\leq x^{1/8-\varepsilon}$, with $\varepsilon>0$, and partition the finite sum.\\

Applying Hypothesis \ref{hyp2225P.560} to the first term yields the upper bound
\begin{eqnarray}\label{eq2626LP.230A}
	T_0(a,b,x)&=&\sum_{d^2\leq x_0}\sum_{e^2\leq x_0}\sum_{\substack{p\leq x\\d^2\mid p+a\\e^2\mid p+b}}\mu((p+a)/d^2)\mu((p+b)/e^2)\\
	&\leq&\sum_{q\leq x_0^4}\bigg |\sum_{\substack{p\leq x\\p\equiv f\bmod q}}\mu(p+a) \mu(p+b) \bigg | \nonumber\\ 
	&=&O\left(\frac{x}{(\log x)^c} \right) ,
\end{eqnarray}
where $q=\lcm(d^2,e^2)\leq x_0^4\leq x^{1/2-\varepsilon}$, and $c>1$ is an arbitrary constant.\\

The second term has the upper bound
\begin{eqnarray}\label{eq2626LP.240A}
	T_1(a,b,x)&=&\sum_{d^2\leq x_0}\sum_{x_0<e^2\leq x}\sum_{\substack{p\leq x\\d^2\mid p+a\\e^2\mid p+b}}\mu((p+a)/d^2)\mu((p+b)/e^2)\\
	&\leq&\sum_{x_0<d^2\leq x,}\sum_{x_0<e^2\leq x}\sum_{\substack{p\leq x\\d^2\mid p+a\\e^2\mid p+b}}1  \nonumber\\
	&=&O\left( x^{1-\varepsilon}\right) \nonumber. 
\end{eqnarray}
The third term has the upper bound
\begin{eqnarray}\label{eq2626LP.250A}
	T_2(a,b,x)&=&\sum_{x_0<d^2\leq x}\sum_{e^2\leq x_0}\sum_{\substack{p\leq x\\d^2\mid p+a\\e^2\mid p+b}}\mu((p+a)/d^2)\mu((p+b)/e^2)\\
	&\leq&\sum_{x_0<d^2\leq x,}\sum_{e^2\leq x_0}\sum_{\substack{p\leq x\\d^2\mid p+a\\e^2\mid p+b}}1  \nonumber\\
	&=&O\left( x^{1-\varepsilon}\right)\nonumber.
\end{eqnarray} 
The fourth term has the upper bound
\begin{eqnarray}\label{eq2626LP.260A}
	T_3(a,b,x)&=&\sum_{x_0<d^2\leq x}\sum_{x_0<e^2\leq x}\sum_{\substack{p\leq x\\d^2\mid p+a\\e^2\mid p+b}}\mu((p+a)/d^2)\mu((p+b)/e^2)\\
	&\leq&\sum_{x_0<d^2\leq x,}\sum_{x_0<e^2\leq x_0}\sum_{\substack{p\leq x\\d^2\mid p+a\\e^2\mid p+b}}1  \nonumber\\
	&=&O\left( x^{1-\varepsilon}\right)\nonumber.
\end{eqnarray}  
Summing \eqref{eq2626LP.230A} to \eqref{eq2626LP.260A} completes the proof.
\end{proof}

\section{Hardy-Littlewood-Chowla Conjecture}\label{S5511}
The Hardy-Littlewood-Chowla Conjecture is a hybrid of the Hardy-Littlewood conjecture, see \cite{DA1849}, \cite[Conjecture B]{HL1923}, and the Chowla conjecture, see \cite{CS1965}, \cite{RO2018}. A version of the merged conjectures has the structure described below.

\begin{conj}\label{conj5511.200} Let $x\geq1$ be a large number, let $a_0,a_1,\ldots, a_{k-1}$ be a fixed admissible $k$-tuple, and let $b_0<b_1<\cdots<b_{l-1}$ be a fixed subset of small integers. Then, 
	\begin{equation}\label{eq5511.200}
		\sum_{n\leq x}\Lambda(n+a_0)\Lambda(n+a_1)\cdots \Lambda(n+a_{k-1})\mu(n+b_0)\mu(n+b_1)\cdots\mu(n+b_{l-1})= o(x) \nonumber.
	\end{equation}
\end{conj}
The special case $k=0$ reduces to the Chowla conjecture, and the special case $l=0$ reduces to the Hardy-Littlewood conjecture. The other cases for $(k,l)\ne(0,0)$ has been proved on average for almost all combinations of the admissible $k$-tuple $a_0,a_1,\ldots, a_{k-1}$, and the small subset of integers $b_0<b_1<\cdots<b_{l-1}$, the details of the proof are provided in \cite{LT2021}. \\

The first proof for the case $k=1$ and $l=1$ for any fixed pair of small integers $a_0=0$ and $b_0\ne0$ is proved in Theorem \ref{thm5757.500}, unconditionally.

\subsection{Twisted vonMangoldt Function over the Squarefree Shifted Primes}\label{S5577}
\begin{lem}\label{lem5577.800} Let $x\geq1$ be a large number, and let $a\ne0$ be a fixed integer. Then, 
	\begin{equation}\label{eq5577.800}
		\sum_{n\leq x}\Lambda(n)\mu^{2}(n+a)= s_0x+O\left(\frac{x}{(\log x)^{c}}\right) \nonumber,
	\end{equation}
	where $s_0=0.373955\ldots$, and $c>1$ is a constant.
\end{lem}

\begin{proof}[\textbf{Proof}] Substituting the identity $\mu(n)^2=\sum_{d^2\mid n}\mu(d)$, and switching the order of summation yield
	\begin{eqnarray}   \label{eq5577.810}
		\sum_{n\leq x}\Lambda(n)\mu^{2}(n+a)
		&=&\sum_{n\leq x}\Lambda(n)\sum_{d^2\mid n+a}\mu(d)\\
		&=&\sum_{d^2\leq x}\mu(d)\sum_{\substack{n\leq x\\ d^2\mid n+a}}\Lambda(n)\nonumber\\
		&=&\sum_{d^2\leq x_0}\mu(d)\sum_{\substack{n\leq x\\ d^2\mid n+a}}\Lambda(n)+\sum_{x_0<d^2\leq x}\mu(d)\sum_{\substack{n\leq x\\ d^2\mid n+a}}\Lambda(n)\nonumber,
	\end{eqnarray}
	where $x_0=(\log x)^{2c_1}$, with $c_1>1$. Applying the Siegel-Walfisz theorem, see \cite[p.\ 405]{FI2010}, \cite[Theorem 15.3]{DL2012}, et cetera, to the first subsum in the above partition yields
	\begin{eqnarray}\label{eq5577.820}
		\sum_{d^2 \leq x_0}\mu(d) 
		\sum_{\substack{n\leq x\\d^2\mid n+a}}\Lambda(n)&=&\sum_{d^2 \leq x_0}\mu(d) \left(\frac{x}{\varphi(d^2)}+ O\left(xe^{-c_0\sqrt{\log x}}\right)\right)\\
		&=&s_0x +O\left(xe^{-c_1\sqrt{\log x}}\right) \nonumber,
	\end{eqnarray}
	where $c_0,c_2>0$ are absolute constants, and 
\begin{equation}\label{eq5577.830}
s_0=	\sum_{n \geq 1} \frac{\mu(n)}{\varphi(n^2)}=\prod_{p\geq2}\left(1-\frac{1}{p(p-1)} \right) 0.373955838964330040631201\ldots, 
\end{equation}	
the same constant appears in \eqref{eq8009MP.390A}. An estimate of the second subsum in the partition yields
	\begin{eqnarray}\label{eq5577.840}
		\sum_{x_0< d^2 \leq x}\mu(d) 
		\sum_{\substack{p\leq x\\d^2\mid p+a}}\Lambda(n)&\leq&(\log x)\sum_{x_0< d^2 \leq x,} \sum_{\substack{p\leq x\\d^2\mid p+a}}1\\
		&\ll&x(\log x)\sum_{x_0<d^2 \leq x}\frac{1}{d^2}\nonumber\\
		&\ll&\frac{x}{(\log x)^{c_1-1}}\nonumber.
	\end{eqnarray}
	Summing \eqref{eq5577.810} and \eqref{eq5577.820}, and setting $c=c_1-1>0$ complete the verification.
	
\end{proof}

\subsection{Signed vonMangoldt Function over the Shifted Primes}\label{S5599}
Define the pattern indicator function 
\begin{equation}\label{eq5599.500A}
	\mu^{+}(n)=\mu^2(n)\left( \frac{1+\mu(n)}{2}\right)=
	\begin{cases}
		1&\text{ if } \mu(n)=1,\\
		0&\text{ if } \mu(n)\ne1,\\
	\end{cases}
\end{equation}
and 
\begin{equation}\label{eq5599.500B}
	\mu^{-}(n)=\mu^2(n)\left( \frac{1-\mu(n)}{2}\right)=
	\begin{cases}
		1&\text{ if } \mu(n)=-1,\\
		0&\text{ if } \mu(n)\ne-1.\\
	\end{cases}
\end{equation}
These are the same as the characteristic functions stated in \eqref{eq8833LN.100}.
\begin{lem}\label{lem5599.510} Let $x\geq1$ be a large number, and let $a\ne0$ be a fixed integer. Then, 
	\begin{equation}\label{eq5599.510}
		\sum_{n\leq x}\Lambda(n)\mu^{\pm}(n+a)= \frac{1}{2}s_0x+\frac{1}{2}\sum_{n\leq x}\Lambda(n)\mu(n+a)+O\left(\frac{x}{(\log x)^{c}}\right) \nonumber,
	\end{equation}
	where $s_0>0$ is a constant, and $c>1$ is an absolute constant.
\end{lem}
\begin{proof}[\textbf{Proof}] Without loss in generality, consider the pattern $\mu^{+}(n+a)=+1$. Now, use the indicator function \eqref{eq5599.500A} to express the single-sign pattern counting function as 
	\begin{eqnarray}   \label{eq5599P.520}
	2\sum_{n\leq x}\Lambda(n)\mu^{+}(n+a)
		&=&\sum_{n\leq x}\Lambda(n)\mu^2(n+a)\left( 1+\mu(n+a)\right)\\
		&=&\sum_{n\leq x}\Lambda(n)\left( \mu^2(n+a)+\mu^3(n+a)\right)\nonumber\\
		&=&\sum_{n\leq x}\Lambda(n) \mu^2(n+a)+\sum_{n\leq x}\Lambda(n) \mu(n+a) \nonumber\\
		&\geq&0\nonumber.
	\end{eqnarray}
	By Lemma \ref{lem5577.800}, the first finite sums has asymptotic evaluation
	\begin{equation}\label{eq5599P.530}
		\sum_{n\leq x}\Lambda(n) \mu^2(n+a)=s_0x+O\left(xe^{-c_0\sqrt{\log x}}\right)
	\end{equation}	
	where $s_0>0$ is a constant, and $c_0>0$ is an absolute constant. Summing these evaluations or estimates verifies the claim for $\sum_{n\leq x}\Lambda(n)\mu^{+}(n+a)\geq0$. The verifications for the other single-sign pattern counting function $\sum_{n\leq x}\Lambda(n)\mu^{-}(n+a)\geq0$ is similar.
\end{proof}

\subsection{Twisted vonMangoldt Function over the Shifted Primes}\label{S5757}
The first case of the hybrid finite, confer Conjecture  \eqref{conj5511.200}, involving the vonMangoldt function and the Mobius function is verified here.
\begin{thm}\label{thm5757.500} Let $x\geq1$ be a large number, and let $a\ne0$ be a fixed integer. Then, 
	\begin{equation}\label{eq5757.500}
		\sum_{n\leq x}\Lambda(n)\mu(n+a)=O\left(\frac{x}{(\log x)^{c}}\right) ,
	\end{equation}
where $c>1$ is an arbitrary constant.
\end{thm}

\begin{proof}[\textbf{Proof}] Rewrite it in terms of the Mobius function over the shifted primes $N(x)=\sum_{p\leq x}\mu(p+a)=O\left(x(\log x) ^{-c_0}\right) $ computed in Theorem \ref{thm1212MP.200}, as
	\begin{eqnarray}   \label{eq5757.510}
		\sum_{n\leq x}\Lambda(n)\mu(n+a)
		&=&\sum_{p\leq x}(\log p)\mu(p+a)\\
		&=&\int_1^x (\log t)d N(t)\nonumber\\
		&=&(\log x)O\left(\frac{x}{(\log x)^{c_0}}\right)-\int_1^x \frac{N(t)}{t}dt\nonumber\\
		&=&O\left(\frac{x}{(\log x)^{c_0-1}}\right)\nonumber,
	\end{eqnarray}
	where $c=c_0-1>1$ is an arbitrary constant.
	
\end{proof}

For an interesting application of Theorem \ref{thm5757.500}, set $a=2k$. Then, a result in \cite[Theorem 1.1]{MV2017} proves that \eqref{eq5757.500} and 
\begin{equation}\label{eq5599.520}
	\sum_{n\leq x}\Lambda(n)\Lambda(n+2k)=\mathfrak{G}(2k)x+o(x) 
\end{equation}  
are equivalent.\\

Consider the single-sign pattern natural density defined by
\begin{equation}\label{eq5599.530}
	\delta_{0}^{\pm}(a)=  \lim_{x\to \infty}\;	\frac{\#\{n\leq x:\Lambda(n)\mu(n+a)=\pm1\log p\}}{x}.
\end{equation}

\begin{thm}\label{thm5757.900} Let $x\geq1$ be a large number, and let $a\ne0$ be a fixed integer. Then, the values $ n\in [1,x]$ such that
	\begin{equation}\label{eq5757.900}
 \Lambda(n)\mu(n+a)=\log p\nonumber
	\end{equation}
are equidistributed on the interval $[1,x]$ as $x\to \infty$. A similar result is valid for $ \Lambda(n)\mu(n+a)=-\log p$. 
\end{thm}

\begin{proof} The single-sign pattern counting function stated in Lemma \ref{lem5599.510} leads to an analytic expression for the natural density
\begin{eqnarray}\label{eq5599.540}
		\delta_{0}^{\pm}(a)&=&  \lim_{x\to \infty}\;	\frac{\#\{n\leq x:\Lambda(n)\mu(n+a)=\pm1\log p\}}{x}\\
		&=& \lim_{x\to \infty}\;	\frac{s_0x+\sum_{n\leq x}\Lambda(n)\mu(n+a)+O\left(x(\log x)^{-c}\right)}{2x}\nonumber\\
		&=&\frac{1}{2}s_0\nonumber.
	\end{eqnarray} 
Here, the last line follows from Theorem \ref{thm5757.500}, and $s_0=0.3739\ldots.$
\end{proof}


\section{Results for the Mobius Autocorrelation Function over the Integers}\label{S7979M}
\subsection{Introduction} \label{S7979MN}
The current estimate of the logarithmic average order of the autocorrelation of the Mobius function $\mu$ has the asymptotic formula
\begin{equation} \label{eq7979MN.100}
	\sum_{n \leq x} \frac{\mu(n) \mu(n+t)}{n} =O \left (\frac{\log x}{\sqrt{\log \log x} } \right ),
\end{equation} 
where $t \ne0$ is a fixed parameter, in \cite[Corollary 1.5]{HR2021} and \cite[Corollary 2]{HH2022}. This improves the estimate $O((\log x)(\log \log \log x)^{-c})$, where $c>0$ is a constant, described in \cite[p. 5]{TT2015}. Here, the following result is considered.   

\begin{thm} \label{thm7979MN.200} Let $\mu:\mathbb{N} \longrightarrow \{-1,0,1\}$ be the Mobius function. Then, for any sufficiently large number $x>1$, and a fixed integer $t \ne0$,
	\begin{equation} \label{eq7979MN.200}
		\sum_{n \leq x} \mu(n) \mu(n+t) =O\left(\frac{x}{(\log \log x)^{1/2-\varepsilon}} \right), 
	\end{equation}	
	where $\varepsilon>0$ is arbitrarily small. 
\end{thm}	

Section \ref{S8009MN} to Section \ref{S8844MN} cover the basic standard materials in analytic number theory, and new materials. The proof of Theorem \ref{thm7979MN.200} is assembled in Section \ref{S3322MN}.

\subsection{Nonlinear Autocorrelation Functions Results}\label{S8009MN}
The number of squarefree integers have the following asymptotic formulas.
\begin{lem} \label{lem9339MN.107} Let $\mu: \mathbb{Z} \longrightarrow \{-1,0,1\}$ be the Mobius function. Then, for any sufficiently large number $x\geq1$, 
	\begin{equation} 
		\sum_{n \leq x} \mu^2(n) =\frac{6}{\pi^2}x+O \left (x^{1/2} \right ). \nonumber
	\end{equation} 
\end{lem}
\begin{proof} Use identity $\mu(n)^2=\sum_{d^2\mid n}\mu(d)$, and other elementary routines, or confer to the literature.\end{proof}

The constant coincides with the density of squarefree integers. Its approximate numerical value is
\begin{equation}\label{eq9339MN.72}
	\frac{6}{\pi^2}=\prod_{p\geq 2}\left ( 1-\frac{1}{p^2}\right )=0.607988295164627617135754\ldots,
\end{equation}
where $p\geq2$ ranges over the primes. The remainder term
\begin{equation} 
	E(x)=\sum_{n \leq x} \mu^2(n) -\frac{6}{\pi^2}x
\end{equation} 
is a topic of current research, its optimum value is expected to satisfies the upper bound $E(x)=O(x^{1/4+\varepsilon})$ for any small number $\varepsilon>0$. Currently, $E(x)=O\left (x^{1/2}e^{-\sqrt{\log x}}\right )$ is the best unconditional remainder term.\\

Assuming $t\ne0$, the earliest result for the autocorrelation of the squarefree indicator function $\mu^2(n)$ appears to be
\begin{equation}\label{eq8009MN.100}
	\sum_{n \leq x}\mu^2(n) \mu^2(n+t)=cx+O\left (x^{2/3} \right ),
\end{equation}
where $c>0$ is a constant, this is proved in \cite{ML1947}.  Except for minor adjustments, the generalization to the $k$-tuple autocorrelation function has nearly the same structure.

\begin{thm}\label{thm8009MN.200}  Let $ q\ne0$, $a_0,a_1, \ldots,a_{k-1}$ be small integers, such that $0\leq a_0<a_1<\cdots<a_{k-1}$. Let $x\geq 1$ be a large number, and let $\mu: \mathbb{Z} \longrightarrow \{-1,0,1\}$ be the Mobius function. Then, 
	\begin{equation}\label{eq8009MN.200}
		\sum_{n \leq x}\mu^2(n+a_0)\mu^2(n+a_1)\cdots \mu^2(n+a_{k-1})=\prod_{p\geq 2}\left ( 1-\frac{\varpi(p)}{p^2}\right )x+O\left (x^{2/3+\varepsilon} \right ),\nonumber
	\end{equation}
	where 
	\begin{equation}\label{eq8009MN.210}
		\varpi(p)=\#\{m\leq p^2: qm+a_i\equiv 0 \bmod p^2 \text{ for } i=0,1,2, ..., k-1\}.
	\end{equation}
	The small number $\varepsilon>0$ and the implied constant depends on $q\ne0$.
\end{thm}
\begin{proof}Consult \cite{ML1947}, \cite[Theorem 1.2]{MI2017}, and the literature.\end{proof}

\begin{exa}\label{exa8009MN.250}{\normalfont
		For the parameters $q=2$, $a_0=0$, and $a_1=1$, the number of solutions of the system of equations is $\varpi(p)=\#\{m\leq p^2: qm+a_i\equiv 0 \bmod p^2=2$ for any prime $p\geq 2$, so constant $s_1=s_1(t)$ has the numerical value, (using $p\leq 10^5$),
		\begin{equation}\label{eq8009MN.220}
			s_1=\prod_{p\geq 2}\left ( 1-\frac{\varpi(p)}{p^2}\right )=\prod_{p\geq 2}\left ( 1-\frac{2}{p^2}\right )=0.32263461660543396347\ldots.
		\end{equation}
	}
\end{exa}

\begin{lem}\label{lem8009MN.350}  Let $x\geq 1$ be a large number, and let $\mu: \mathbb{Z} \longrightarrow \{-1,0,1\}$ be the Mobius function. If $t\ne0$ is a fixed integer,  then, 
	\begin{equation}\label{eq8009MN.350}
		\sum_{n \leq x}\mu(n)^2 \mu(n+t)=O\left (\frac{x}{(\log x)^c} \right ),\nonumber
	\end{equation}
	where $c> 0$ is an arbitrary constant.
\end{lem}
\begin{proof} Substitute the identity $\mu(n)^2=\sum_{d^2\mid n}\mu(d)$, and switching the order of summation yield
	\begin{eqnarray}\label{eq8009MN.360B}
		\sum_{n \leq x}\mu(n)^2 \mu(n+t)&=&\sum_{n \leq x} \mu(n+t)\sum_{d^2\mid n}\mu(d)\\
		&=&\sum_{d^2 \leq x}\mu(d) \sum_{\substack{n\leq x\\d^2\mid n}}\mu(n+t)\nonumber\\
		&=&\sum_{d^2 \leq x^{2\varepsilon}}\mu(d) \sum_{\substack{n\leq x\\d^2\mid n}}\mu(n+t)+\sum_{x^{2\varepsilon}< d^2 \leq x}\mu(d) \sum_{\substack{n\leq x\\d^2\mid n}}\mu(n+t)\nonumber,
	\end{eqnarray}
	where $\varepsilon\in (0,1/4)$. Applying Corollary \ref{cor2225P.550} to the first subsum in the partition yields
	\begin{eqnarray}\label{eq8009MN.370B}
		\sum_{d^2 \leq x^{2\varepsilon}}\mu(d) \sum_{\substack{n\leq x\\d^2\mid n}}\mu(n+t)&\leq&\sum_{q \leq x^{\varepsilon}} \bigg |\mu(d) \sum_{\substack{n\leq x\\m\equiv b \bmod q}}\mu(m)\bigg |\\
		&=&O\left( \frac{x}{(\log x)^{c}}\right) \nonumber,
	\end{eqnarray}
	where $q=d^2$. An estimate of the second subsum in the partition yields
	\begin{eqnarray}\label{eq8009MN.380B}
		\sum_{x^{2\varepsilon}< d^2 \leq x}\mu(d) \sum_{\substack{n\leq x\\d^2\mid n}}\mu(n+t)&\leq&\sum_{x^{2\varepsilon}< d^2 \leq x} \sum_{\substack{n\leq x\\d^2\mid n}}1\\
		&\ll&x\sum_{d^2 \leq x}\frac{1}{d^2}\nonumber\\
		&\ll&x^{1-\varepsilon}\nonumber.
	\end{eqnarray}
	Summing \eqref{eq8009MN.370B} and \eqref{eq8009MN.380B} completes the verification.
\end{proof}

\subsection{Logarithm Average and Arithmetic Average Connection} \label{S7766MN}
Let $f: \N \longrightarrow \C$ be an arithmetic function. The connection between the logarithm average 
\begin{equation} \label{eq7766MN.050}
	\sum_{n \leq x} \frac{f(n) }{n} =M_L(x)+E_L(x)
\end{equation}
and the arithmetic average 
\begin{equation} \label{eq7766MN.060}
	\sum_{n \leq x} f(n) =M_A(x)+E_A(x),
\end{equation}
where $M_i(x)$ and $E_i(x)$ are the main terms and error terms respectively, is important in partial summations. The error term $E_L(x)$ required to compute an effective arithmetic average \eqref{eq7766MN.060} directly from the logarithm average \eqref{eq7766MN.050} is explained in \cite[Section 2.12]{HA2013}, see also \cite[Exercise 2.12]{HA2013}.

\begin{lem}\label{lem7766MN.100} Let $t\ne0$ be a small integer, and let $x\geq1$ be a large number. An effective nontrivial arithmetic average $\sum_{n \leq x} \mu(n) \mu(n+t)=o(x)$ can be computed directly from the logarithm average $A(x)=\sum_{n \leq x} \mu(n) \mu(n+t)n^{-1}$ if and only if $A(x)=O(1/f(x))$, where $f(x)$ is a monotonically increasing function. 
\end{lem}
\begin{proof} Assume it is nontrivial. By partial summation, 
	\begin{eqnarray}\label{eq7766MN.110}
		o(x)=\sum_{n \leq x} \mu(n) \mu(n+t)&=&\sum_{n \leq x} n \cdot \frac{\mu(n) \mu(n+t)}{n} 	\\
		&=&\int_1^x z \,dA(z)\nonumber \\
		&=&xA(x)-\int_1^x A(z)dz\nonumber.
	\end{eqnarray}
	This implies that $A(x)=O(x/f(x))$. Conversely. If $A(x)=O(x/f( x))$, then 
	\begin{equation} \label{eq7766MN.120}
		\sum_{n \leq x} \mu(n) \mu(n+t) =o(x)
	\end{equation}
	as claimed.
\end{proof}
The known logarithm average \eqref{eq7979MN.100} is not of the form $A(x)=O(1/f(x))$, but it can be used to compute an \textit{absolute} upper bound. 

\begin{lem}\label{lem7766MN.200} Let $t\ne0$ be a small integer, and let $x\geq1$ be a large number. If the logarithm average $A(x)=\sum_{n \leq x} \mu(n) \mu(n+t)n^{-1}=O(\log x)(\log\log x)^{-1/2}$, then the arithmetic average $\sum_{n \leq x} \mu(n) \mu(n+t)\ll x(\log \log x)^{-1/2+\varepsilon}$, where $\varepsilon>0$ is a small number. 
\end{lem}
\begin{proof} Assume $B(x)=\sum_{n \leq x} \mu(n) \mu(n+t)\gg x(\log \log x)^{-1/2+\varepsilon}$. Then,
	\begin{eqnarray}\label{eq7766MN.210}
		\frac{\log x}{(\log \log x)^{1/2}}&\gg&\sum_{n \leq x} \frac{\mu(n) \mu(n+t)}{n}\\
		&=&\int_1^x \frac{1}{z} \,dB(z)\nonumber \\
		&=&\frac{B(z)}{z}+\int_1^x \frac{B(z)}{z^2}dz\nonumber	.
	\end{eqnarray}
	Since the integral \begin{equation}\label{eq7766MN.220}
		\int_1^x \frac{B(z)}{z^2}dz=	\int_2^x\frac{1}{z(\log \log z)^{1/2-\varepsilon}} dz\gg \frac{\log x}{(\log \log x)^{1/2-\varepsilon}},
	\end{equation}
	for sufficiently large $x\geq1$, the assumption is false. Hence, it implies that there is an absolute upper bound $B(x)\ll x(\log \log x)^{-1/2+\varepsilon}$.
\end{proof}

\subsection{Single-Sign Patterns Mobius Characteristic Functions} \label{S8822MN}
The analysis of single-sign pattern characteristic function
is well known. 

\begin{lem}\label{lem8822MN.200A} If $\mu(n)\in \{-1,1\}$ is the Mobius function, then,
	\begin{eqnarray}\label{eq8822MN.200A}
		\mu^{\pm}(n)&=&\mu^2(n)\left( \frac{1\pm\mu(n)}{2}\right)\\	&=&
		\begin{cases}
			1 &\text{ if } \mu(n)=\pm1,\\
			0 &\text{ if } \mu(n)\ne\pm1,\nonumber\\
		\end{cases}
	\end{eqnarray}
	are the characteristic functions of the subset of primes
	\begin{equation}\label{eq8822MN.210A}
		\mathcal{N}_{\mu}^{\pm }	=	\{n\geq 1: \mu(n)=\pm1\}.
	\end{equation} 
\end{lem}

\subsection{Double-Sign Patterns Characteristic Functions} \label{S8833MN}
The principle of single-sign pattern characteristic function is extended to the double-sign patterns $(\mu(n+a),\mu(n+b))=(\pm1,\pm1)$, where $a,b\in \Z$ is a pair of small integers such that $a\ne b$. Other sign patterns are topics of current research, confer \cite{HA1986}, \cite{KP1986}, \cite[Corollary 1.7]{TT2015}, \cite{MT2015}, and similar literature, for details. A new and different approach to the analysis of double-sign patterns, triple-sign patterns, et cetera, based on elementary methods, is provided here. 

\begin{lem}\label{lem8833MN.200} Let $t\ne0$ be an integer, and let $\mu(n)\in \{-1,0,1\}$ be the Mobius function. Then,
	\begin{eqnarray}\label{eq8833MN.200}
		\mu^{\pm \pm}(n,t)&=&\mu^2(n)\mu^2(n+t)\left( \frac{1\pm\mu(n)}{2}\right)\left( \frac{1\pm\mu(n+t)}{2}\right)\\	&=&
		\begin{cases}
			1 &\text{ if } \mu(n)=\pm1,\mu(n+t)=\pm1,\\
			0 &\text{ if } \mu(n)\ne\pm1,\mu(n+t)\ne\pm1,\nonumber\\
		\end{cases}
	\end{eqnarray}
	are the characteristic functions of the subset of integers
	\begin{equation}\label{eq8833MN.210}
		\mathcal{N}_{\mu}^{\pm \pm}(t)	=	\{n\geq 1: \mu(n)=\pm1, \mu(n+t)=\pm1\}.
	\end{equation} 
\end{lem}

\subsection{Trivial Double-Sign Patterns Counting Functions} \label{S8811MN}
Let $\epsilon_i\in \{-1,0,1\}$, and $(\mu(n),\mu(n+t)=(\epsilon_0,\epsilon_1)$ denote an arbitrary double-sign pattern. The trivial double-sign patterns 
\begin{equation}\label{8811MN.100}
	(-1,0),\quad (0,-1),\quad(0,0),\quad(0,1),\quad(1,0),
\end{equation}
do not contribute to the autocorrelation function. Nevertheless, it is sometimes required to quantify 
the natural densities of these double-sign patterns.
\begin{lem} \label{lem8811MN.200} For a large number $x\geq1$, the zero double-sign pattern $(\epsilon_0,\epsilon_1)=(0,0)$ counting function is given by
\begin{equation}\label{eq8811MN.200}
	R^{00}(t,x)=\sum_{\substack{n\leq x\\ \mu(n)\ne0,\; \mu(n+t)\ne0}}1=s_2x+O\left(x^{2/3} \right)\nonumber, 
\end{equation}
where $s_2=1-2\zeta(2)^{-1}+s_1>0$ is a constant. Further, the natural density of the double zero pattern $(0,0)$ is
\begin{equation}\label{eq8811MN.210}
\delta^{00}(t)=1-2\zeta(2)^{-1}+s_1\nonumber. 
\end{equation}
\end{lem}
\begin{proof}The counting function for the zero double-sign pattern $(\epsilon_0,\epsilon_1)=(0,0)$ is
	\begin{eqnarray}   \label{eq8811MN.220}
	R^{00}(t,x)
	&=&\sum_{n\leq x}\left( 1-\mu^2(n)\right) \left( 1-\mu^2(n+t)\right)\\
	&=&\sum_{n\leq x}\left( 1-\mu^2(n)-\mu^2(n+t)+\mu^2(n)\mu^2(n+t)\right)\nonumber\\
	&=&\sum_{n\leq x}1-\sum_{n\leq x}\mu^2(n)-\sum_{n\leq x}\mu^2(n+t)+\sum_{n\leq x}\mu^2(n)\mu^2(n+t) \nonumber\\
	&\geq&0\nonumber.
\end{eqnarray}
The last four finite sums have the following evaluations or estimates.
\begin{enumerate}
\item $ \displaystyle \sum_{n\leq x}1=[x], $\tabto{8cm}
\item $ \displaystyle \sum_{n\leq x}\mu^2(n)=\zeta(2)^{-1}x+O\left(x^{1/2} \right), $\tabto{8cm}
see Lemma \ref{lem9339MN.107},
\item $ \displaystyle \sum_{n\leq x}\mu^2(n+t)=\zeta(2)^{-1}x+O\left(x^{1/2} \right), $\tabto{8cm}
see Lemma \ref{lem9339MN.107},
\item $ \displaystyle \sum_{n\leq x}\mu^2(n)\mu^2(n+t)=s_1x+O\left(x^{2/3} \right), $\tabto{8cm}
see Theorem \ref{thm8009MN.200},
\end{enumerate}
where $[x]$ is the largest integer function, and $s_1=s_1(t)>0$ is a constant. Summing these evaluations or estimates verifies the claim for $R^{00}(t,x)\geq0$. 
\end{proof}

\begin{lem} \label{lem8811MN.300} For a large number $x\geq1$, the zero double-sign pattern $(\epsilon_0,\epsilon_1)=(0,0)$ counting function is given by
	\begin{equation}\label{eq8811MN.300}
		R^{01}(t,x)=\sum_{\substack{n\leq x\\ \mu(n)=0,\; \mu(n+t)=1}}1=s_2x+O\left(x^{2/3} \right), 
	\end{equation}
	where $s_2=(\zeta(2)^{-1}-s_1)/2>0$, and $c>0$ is an absolute constant. Further, the natural density of the double zero pattern $(0,1)$ \begin{equation}\label{eq8811MN.310}
\delta^{01}(t)=(\zeta(2)^{-1}-s_1)/2\nonumber. 
	\end{equation}
\end{lem}
\begin{proof}The counting function for double-sign pattern $(\epsilon_0,\epsilon_1)=(0,1)$ is
\begin{eqnarray}   \label{eq8811MN.320}
2R^{01}(t,x)
&=&\sum_{n\leq x}\left( 1-\mu^2(n)\right) \left( 1+\mu(n+t)\right)\mu^2(n+t)\\
&=&\sum_{n\leq x}\mu^2(n+t)-\sum_{n\leq x}\mu^2(n)\mu^2(n+t)\nonumber\\
		&&\hskip 1 in +\sum_{n\leq x}\mu^3(n+t)-\sum_{n\leq x}\mu^2(n)\mu^3(n+t) \nonumber\\
		&\geq&0\nonumber.
	\end{eqnarray}
	The last four finite sums have the following evaluations or estimates.
\begin{enumerate}
\item $ \displaystyle \sum_{n\leq x}\mu^2(n+t)=\zeta(2)^{-1}x+O\left(x^{1/2} \right), $\tabto{8cm}see Lemma \ref{lem9339MN.107},
\item $ \displaystyle \sum_{n\leq x}\mu^2(n)\mu^2(n+t)=s_0x+O\left(x^{2/3} \right), $
\tabto{8cm}see Theorem \ref{thm8009MN.200},
\item $ \displaystyle \sum_{n\leq x}\mu^3(n+t)=O\left(xe^{-c\sqrt{\log x}} \right), $
\tabto{8cm}see Theorem \ref{thm2222.500},
\item $ \displaystyle \sum_{n\leq x}\mu^2(n)\mu^3(n+t)=O\left(xe^{-c\sqrt{\log x}} \right), $
\tabto{8cm}see Lemma \ref{lem8009MN.350},
	\end{enumerate}
where $s_0=s_0(t)>0$ is a constant, and $c>0$ is an absolute constant. Summing these evaluations or estimates verifies the claim for $R^{01}(t,x)\geq0$. 
\end{proof}

The same counting function and natural density given in Lemma \ref{lem8811MN.300} applies to any of the trivial double-sign patterns $(0,-1)$, $(1,0)$,$(-1,0)$. The verification is similar, mutatis mutandis. \\

The numerical value of these densities are the followings.
\begin{enumerate}
	\item $\displaystyle s_1= 0.32263461660543396347\ldots$,\tabto{8cm}computed in Example \ref{exa8009MN.250},
	\item $\displaystyle s_2=1-2\zeta(2)^{-1}+s_1=0.106780412897381\ldots,
	$
	\item $\displaystyle s_3=(\zeta(2)^{-1}-s_1)/2=0.142646242624296\ldots .
	$
\end{enumerate}
\subsection{Double-Sign Patterns Counting Functions} \label{S8844MN}
The single-sign patterns $\mu(n)=1$ and $\mu(n)=-1$ have single-sign pattern counting functions of the forms
\begin{equation}\label{eq8844MN.100A}
	R^{+}(x)=\sum_{\substack{n\leq x\\ \mu(n)=1}}1=\sum_{n\leq x}\left( \frac{1+\mu(n)}{2}\right) \mu^2(n)=\frac{1}{2}\frac{x}{\zeta(2)}+O\left( xe^{-c\sqrt{\log x}}\right), 
\end{equation}
and 
\begin{equation}\label{eq8844MN.100B}
	R^{-}(x)=\sum_{\substack{n\leq x\\ \mu(n)=-1}}1=\sum_{n\leq x}\left( \frac{1-\mu(n)}{2}\right) \mu^2(n)=\frac{1}{2}\frac{x}{\zeta(2)}+O\left( xe^{-c\sqrt{\log x}}\right), 
\end{equation}
these follow from Lemma \ref{lem9339MN.107} and Theorem \ref{thm2222.500}. In terms of these functions, the summatory Mobius function has the asymptotic formula
\begin{eqnarray}\label{eq8844MN.100C}
	R(x)&=&\sum_{n\leq x}\mu(n)\\
	&=&R^{+}(x)-R^{-}(x)\nonumber\\
	&=&\left( \left(\frac{1}{2}\frac{x}{\zeta(2)}\right) +O\left( xe^{-c\sqrt{\log x}}\right) \right)-\left( \left(\frac{1}{2}\frac{x}{\zeta(2)}\right) +O\left( xe^{-c\sqrt{\log x}}\right) \right) \nonumber\\
	&=&O\left( xe^{-c\sqrt{\log x}}\right)\nonumber. 
\end{eqnarray}
Basically, it is a different form of the Prime Number Theorem
\begin{equation}\label{eq8844MN.100D}
	\pi(x)=\li(x)+O\left(xe^{-c\sqrt{\log x}} \right) ,	
\end{equation}
where $\li(x)=\int_2^2(\log t)^{-1}dt$ is the logarithm integral, and $c>0$ is an absolute constant, see \cite[Eq.~27.12.5]{DLMF}, \cite[Theorem 3.10]{EL1985}, et alii. \\

The same principle is applied to the double-sign patterns $(\mu(n),\mu(n+t))=(\pm1,\pm1)$ to derive the extended results provided here. \\

The double-sign pattern counting functions are defined by
\begin{equation}\label{eq8844MN.110A}
	R^{++}(t,x)=\sum_{\substack{n\leq x\\ \mu(n)=1,\; \mu(n+t)=1}}1=\sum_{\substack{n\leq x\\ n\in \mathcal{N}_{\mu}^{++}(t)}}1,
\end{equation}
\begin{equation}\label{eq8844MN.110B}
	R^{+-}(t,x)=\sum_{\substack{n\leq x\\ \mu(n)=1,\; \mu(n+t)=-1}}1=\sum_{\substack{n\leq x\\ n\in \mathcal{N}_{\mu}^{+-}(t)}}1,
\end{equation}
\begin{equation}\label{eq8844.110C}
	R^{-+}(t,x)=\sum_{\substack{n\leq x\\ \mu(n)=-1,\; \mu(n+t)=1}}1=\sum_{\substack{n\leq x\\ n\in \mathcal{N}_{\mu}^{-+}(t)}}1,
\end{equation}
\begin{equation}\label{eq8844MN.110D}
	R^{- -}(t,x)=\sum_{\substack{n\leq x\\ \mu(n)=-1,\; \mu(n+t)=-1}}1=\sum_{\substack{n\leq x\\ n\in \mathcal{N}_{\mu}^{- -}(t)}}1.
\end{equation}

The double-sign pattern counting functions \eqref{eq8844MN.110A} to \eqref{eq8844MN.110D} are precisely the cardinalities of the subsets of integers
\begin{multicols}{2}
	\begin{enumerate}
		\item $\mathcal{N}_{\mu}^{++}(t)\subset \N$ ,
		\item $\mathcal{N}_{\mu}^{+-}(t)\subset \N$ ,
		\item $\mathcal{N}_{\mu}^{-+}(t)\subset \N$ ,
		\item $\mathcal{N}_{\mu}^{--}(t)\subset \N$ ,
	\end{enumerate}
\end{multicols}
defined in \eqref{eq8833MN.210}. 
The next result is required to complete the analysis of the asymptotic formula for $R(t,x)$, which is completed in the next section.

\begin{lem}\label{lem8844MN.300} Let $x\geq1$ be a large number, and let $t\ne0$ be a fixed integer. Then, 
	\begin{equation}\label{eq8844MN.300}
		R^{\pm \pm}(t,x)= \frac{1}{4}s_1x+ \frac{1}{4}R(t,x)+O\left(\frac{x}{(\log x)^{c}} \right) ,\nonumber
	\end{equation}
	where $s_1=s_1(t)>0$, and $c>0$ are constants.
\end{lem}

\begin{proof} Without loss in generality, consider the double-sign pattern $(\mu(n),\mu(n+t))=(+1,+1)$. Now, use Lemma \ref{lem8833MN.200} to express the double-sign pattern counting function as 
	\begin{eqnarray}   \label{eq4422MN.310}
		4R^{++}(t,x)&=&\sum_{n\leq x}\mu^{++}(t,n)\\
		&=&\sum_{n\leq x}\mu^2(n)\mu^2(n+t)\left( 1+\mu(n)\right) \left( 1+\mu(n+t)\right)\nonumber\\
		&=&\sum_{n\leq x}\mu^2(n)\mu^2(n+t)\left( 1+\mu(n)+\mu(n+t)+\mu(n)\mu(n+t)\right)\nonumber\\
		&=&\sum_{n\leq x}\mu^2(n)\mu^2(n+t)+\sum_{n\leq x}\mu^3(n)\mu^2(n+t) \nonumber\\
		&&\hskip 1 in +\sum_{n\leq x}\mu(n)^2\mu^3(n+t)+\sum_{n\leq x}\mu^3(n)\mu^3(n+t)\nonumber\\
		&\geq&0\nonumber.
	\end{eqnarray}
	The last four finite sums have the following evaluations or estimates.
	\begin{enumerate}
		\item $ \displaystyle \sum_{n\leq x}\mu^2(n)\mu^2(n+t)=s_0(t)x+O\left(x^{2/3} \right), $\tabto{8cm}see Theorem \ref{thm8009MN.200},		
		\item $ \displaystyle \sum_{n\leq x}\mu^3(n)\mu^2(n+t)=O\left(\frac{x}{(\log x)^{c}} \right), $\tabto{8cm}see Lemma \ref{lem8009MN.350},		
		\item $ \displaystyle \sum_{n\leq x}\mu^2(n)\mu^3(n+t)=O\left(\frac{x}{(\log x)^{c}} \right), $\tabto{8cm}
		see Lemma \ref{lem8009MN.350},		
		\item $ \displaystyle \sum_{n\leq x}\mu^3(n)\mu^3(n+t)=\sum_{n\leq x}\mu(n)\mu(n+t), $\tabto{8cm}since $\mu^{2k+1}(n)=\mu(n)$ for $k\geq0$. 
	\end{enumerate}
where $s_1=s_1(t)>0$ is a constant, and $c>0$ is a constant. 
	Summing these evaluations or estimates verifies the claim for $R^{++}(t,x)\geq0$. The verifications for the next three double-sign pattern counting functions $R^{+-}(t,x)\geq0$, $R^{-+}(t,x)\geq0$, and $R^{--}(t,x)\geq0$ are similar.
\end{proof}

\subsection{Equidistribution of Double-Sign Patterns}\label{S5225MN}
The nontrivial result for the summatory Mobius function
\begin{eqnarray}\label{eq5225MN.200}
	\sum_{n\leq x}\mu(n)
	&=&R^{+}(x)-R^{-}(x)=O\left( xe^{-c\sqrt{\log x}}\right)
\end{eqnarray}
has no main term. It vanishes because the number of single-sign patterns $R^{+}(x)=\#\{n\leq x: \mu(n)=1\}$ and $R^{-}(x)=\#\{n\leq x: \mu(n)=-1\}$ have the same cardinality. This implies that the single-sign patterns are equidistributed on the interval $[1,x]$, and each has the natural density $\delta_{\mu}^{\pm}=3/\pi^2$, see \eqref{eq8844MN.100A} for more detail. This idea is extended in the proof of the equidistribution of the double-sign patterns. \\

Recall that $R^{\pm\pm}(t,x)=\#\{n\leq x:\mu(n)=\pm1 ,\mu(n+t)=\pm1\}$, and the natural density of a double-sign pattern is defined by
\begin{equation}\label{eq5225MN.500}
	\delta_{\mu}^{\pm\pm}(t)=  \lim_{x\to \infty}\;	\frac{\#\{n\leq x:\mu(n)=\pm1 ,\mu(n+t)=\pm1\}}{x}.
\end{equation}  

\begin{thm}\label{thm5225MN.500} Let $x\geq1$ be a large number, and let $t\ne0$ be a fixed integer. Then, the double-sign patterns $++$, $+-$, $-+$, and $--$ of the Mobius pair $\mu(n), \mu(n+t)$ are equidistributed on the interval $[1,x]$. In particular, each double-sign pattern has the natural density 
	\begin{equation}\label{eq5225MN.510}
		\delta_{\mu}^{\pm\pm}(t)= \frac{1}{4}s_1,\nonumber
	\end{equation}
	where $s_1=s_1(t)>0$ is a constant.
\end{thm}
\begin{proof} By Lemma \ref{lem7766MN.200}, $R(t,x)=o(x)$. Accordingly, the limit of the proportion of double-sign pattern
	\begin{eqnarray}\label{eq8844MN.550}
		\delta_{\mu}^{\pm\pm}(t)&=&  \lim_{x\to \infty}\;	\frac{\{n\leq x:\mu(n)=\pm1 ,\mu(n+t)=\pm1\}}{x}\nonumber\\
		&=&\lim_{x\to \infty}\;	\frac{s_1x+R(t,x)+O(x(\log x)^{-c})}{4x}\nonumber\\
		&=&\frac{1}{4}s_1\nonumber.
	\end{eqnarray} 
\end{proof}

This proves that the double-sign patterns $++$, $+-$, $-+$, and $--$ are equidistributed on the interval $[1,x]$
as $x\to \infty$. 
\begin{exa}\label{exa5225MN.600}{\normalfont Let $t=1$. The constant $s_1=s_1(1)=0.322634\ldots$ for the double-sign patterns $\mu(n)=\pm1, \mu(n+1)=\pm1$ is computed in Example \ref{exa8009MN.250}. Thus, by Theorem \ref{thm5225MN.500}, in any sufficiently large interval $[1,x]$, the number of double-sign patterns
		\begin{equation}\label{eq5225MN.610}
			R^{\pm\pm}(t,x)	=\delta_{\mu}^{\pm\pm}(t)x+O\left( \frac{x}{\log \log x)^{1/2-\varepsilon}}\right) 
		\end{equation}		
		are the followings.
		\begin{enumerate}
\item$ \displaystyle R^{++}(t,x)=\frac{0.3226\ldots}{4}x+O\left( \frac{x}{\log \log x)^{1/2-\varepsilon}}\right)$, for $\mu(n)=1, \mu(n+1)=1,$
\item$ \displaystyle R^{+-}(t,x)=\frac{0.3226\ldots}{4}x+O\left( \frac{x}{\log \log x)^{1/2-\varepsilon}}\right)$,
			for $\mu(n)=1, \mu(n+1)=-1,$
\item$ \displaystyle R^{-+}(t,x)=\frac{0.3226\ldots}{4}x+O\left( \frac{x}{\log \log x)^{1/2-\varepsilon}}\right)$, for $\mu(n)=-1, \mu(n+1)=1,$			
			\item$ \displaystyle R^{--}(t,x)=\frac{0.3226\ldots}{4}x+O\left( \frac{x}{\log \log x)^{1/2-\varepsilon}}\right)$, for $\mu(n)=-1, \mu(n+1)=-1.$			
		\end{enumerate}
		Consequently, the main term of the autocorrelation function
		\begin{eqnarray}\label{eq5225MN.620}
			\sum_{n\leq x}\mu(n)\mu(n+1)
			&=&R^{++}(t,x)-R^{+-}(t,x)+R^{--}(t,x)-R^{-+}(t,x)\nonumber\\
			&=&O\left( x(\log \log x)^{-1/2+\varepsilon}\right)	.
		\end{eqnarray}
		vanished. For $x=10^4$, the actual value of the autocorrelation function is
		\begin{equation}\label{eq5225MN.630}
			\sum_{n\leq x}\mu(n)\mu(n+1)
			=12,
		\end{equation}
		and the actual values of the double-sign counting functions are tabulated below.
		\begin{center}
			\begin{tabular}{ c|c|c } 
				$\mu(n)$ & $\mu(n+1)$ & $R^{\pm\pm}(1,x)$ \\
				\hline 
				$+1$ & $+1$ & $3228/4$ \\ 
				$+1$ & $-1$ & $3152/4$ \\ 
				$-1 $& $+1 $& $3282/4$ \\
				$-1$ &$ -1 $& $3256/4$ \\		
			\end{tabular}
		\end{center}
		The differences among the double-sign counting functions $R^{\pm\pm}(1,x)$ seem to be properties of the biases toward the different double-sign patterns.

	}
\end{exa}

\section{Result for the Mobius Autocorrelation Function over the Integers} \label{S3322MN}
The elementary results presented in the previous sections are utilized here to prove an effective form of the autocorrelation function $\sum_{n\leq x}\mu(n)\mu(n+t)$, which is a significant improvement over the current results in the literature, see \eqref{eq7979MN.100}. The average result claims that
\begin{equation}\label{eq3322MN.400}
	\sum_{t\leq T}	\sum_{n\leq x}\mu(n)\mu(n+t)=o(Tx),
\end{equation}
see \cite[Theorem 1.1]{MR2015} for the precise details, and the conjecture due to Chowla, see \cite{CS1965}, and \cite{RO2018}, claims that 
\begin{equation}\label{eq3322MN.410}
	\sum_{n\leq x}\mu(n)\mu(n+t)=o(x),
\end{equation}
for any small fixed integer $t\ne0$.

\begin{proof}[\textbf{Proof:} {\normalfont \textbf{(Theorem \ref{thm7979MN.200})}}] By Theorem \ref{thm5225MN.500}, each double-sign pattern has the asymptotic formula
\begin{equation}\label{eq3322MN.720}
	R^{\pm\pm}(t,x)	=\frac{1}{4}s_1x+O\left( x(\log \log x)^{-1/2+\varepsilon}\right).
\end{equation}	
In terms of the previous double-sign pattern counting functions, the Mobius autocorrelation function has form 
	\begin{eqnarray}\label{eq3322MN.700}
	R(t,x)&=&\sum_{n\leq x}\mu(n)\mu(n+t)\\
	&=&R^{++}(t,x)-R^{+-}(t,x)+R^{--}(t,x)-R^{-+}(t,x)\nonumber\\
	&=&\;\;\frac{1}{4}s_1(t)x+ O\left( x(\log \log x)^{-1/2+\varepsilon}\right) -\frac{1}{4}s_1(t)x+ O\left( x(\log \log x)^{-1/2+\varepsilon}\right) \nonumber\\
	&& + \frac{1}{4}s_1(t)x+ O\left( x(\log \log x)^{-1/2+\varepsilon}\right)  -\frac{1}{4}s_1(t)x+ O\left( x(\log \log x)^{-1/2+\varepsilon}\right)\nonumber\\
	&=& O\left( x(\log \log x)^{-1/2+\varepsilon}\right) \nonumber,	 		
\end{eqnarray}	
where $\varepsilon>0$ is an arbitrarily small number.	
Quod erat demonstrandum.
\end{proof}

\section{Results for the Liouville Autocorrelation Function over the Integers}\label{S7979LN}
The current estimate of the logarithmic average order of the autocorrelation of the Liouville function $\lambda$ has the asymptotic formula
\begin{equation} \label{eq7766LN.100}
	\sum_{n \leq x} \frac{\lambda(n) \lambda(n+t)}{n} =O \left (\frac{\log x}{\sqrt{\log \log x} } \right ),
\end{equation} 
where $t \ne0$ is a fixed parameter, in \cite[Corollary 1.5]{HR2021} and \cite[Corollary 2]{HH2022}. This improves the estimate $O((\log x)(\log \log \log x)^{-c})$, where $c>0$ is a constant, described in \cite[p. 5]{TT2015}. Here, the following result is considered. 

\begin{thm} \label{thm7979LN.200} Let $\lambda:\mathbb{N} \longrightarrow \{-1,1\}$ be the Liouville function, and let $a,b \in \Z$ be a pair of fixed integers such that $a\ne b$. Then, for any sufficiently large number $x>1$, 
	\begin{equation} \label{eq7979L.200}
		\sum_{n \leq x} \lambda(n+a) \lambda(n+b) =O\left(xe^{-c\sqrt{\log x}} \right), 
	\end{equation}	
where $c>0$ is an absolute constant. 
\end{thm}	

Subsection \ref{S7766LN} to Subsection \ref{S8877LN} cover basic materials in analytic number theory, and new materials. The proof of Theorem \ref{thm7979LN.200} is spliced together in Subsection \ref{S5225LN}. 

\subsection{Logarithm Average and Arithmetic Average Connection} \label{S7766LN}
Let $f: \N \longrightarrow \C$ be an arithmetic function. The connection between the logarithm average 
\begin{equation} \label{eq7766LN.050}
	\sum_{n \leq x} \frac{f(n) }{n} =M_L(x)+E_L(x)
\end{equation}
and the arithmetic average 
\begin{equation} \label{eq7766LN.060}
	\sum_{n \leq x} f(n) =M_A(x)+E_A(x),
\end{equation}
where $M_i(x)$ and $E_i(x)$ are the main terms and error terms respectively, is important in partial summations. The error term $E_L(x)$ required to compute an effective arithmetic average \eqref{eq7766LN.060} directly from the logarithm average \eqref{eq7766LN.050} is explained in \cite[Section 2.12]{HA2013}, see also \cite[Exercise 2.12]{HA2013}.

\begin{lem}\label{lem7766LN.100} Let $t\ne0$ be a small integer, and let $x\geq1$ be a large number. An effective nontrivial arithmetic average $\sum_{n \leq x} \lambda(n) \lambda(n+t)=o(x)$ can be computed directly from the logarithm average $A(x)=\sum_{n \leq x} \lambda(n) \lambda(n+t)n^{-1}$ if and only if $A(x)=O(1/f(x))$, where $f(x)$ is a monotonically increasing function. 
\end{lem}
\begin{proof} Assume it is nontrivial. By partial summation, 
	\begin{eqnarray}\label{eq7766LN.110}
		o(x)=\sum_{n \leq x} \lambda(n) \lambda(n+t)&=&\sum_{n \leq x} n \cdot \frac{\lambda(n) \lambda(n+t)}{n} 	\\
		&=&\int_1^x z \,dA(z)\nonumber \\
		&=&xA(x)-\int_1^x A(z)dz\nonumber.
	\end{eqnarray}
	This implies that $A(x)=O(x/f(x))$. Conversely. If $A(x)=O(x/f( x))$, then 
	\begin{equation} \label{eq7766N.120}
		\sum_{n \leq x} \lambda(n) \lambda(n+t) =o(x)
	\end{equation}
	as claimed.
\end{proof}
The known logarithm average \eqref{eq7766LN.100} is not of the form $A(x)=O(1/f(x))$, but it can be used to compute an \textit{absolute} upper bound. 

\begin{lem}\label{lem7766LN.200} Let $t\ne0$ be a small integer, and let $x\geq1$ be a large number. If the logarithm average $A(x)=\sum_{n \leq x} \lambda(n) \lambda(n+t)n^{-1}=O(\log x)(\log\log x)^{-1/2}$, then the arithmetic average $\sum_{n \leq x} \lambda(n) \lambda(n+t)\ll x(\log \log x)^{-1/2+\varepsilon}$, where $\varepsilon>0$ is a small number. 
\end{lem}
\begin{proof} Assume $B(x)=\sum_{n \leq x} \lambda(n) \lambda(n+t)\gg x(\log \log x)^{-1/2+\varepsilon}$. Then,
	\begin{eqnarray}\label{eq7766LN.210}
		\frac{\log x}{(\log \log x)^{1/2}}&\gg&\sum_{n \leq x} \frac{\lambda(n) \lambda(n+t)}{n}\\
		&=&\int_1^x \frac{1}{z} \,dB(z)\nonumber \\
		&=&\frac{B(z)}{z}+\int_1^x \frac{B(z)}{z^2}dz\nonumber	.
	\end{eqnarray}
	Since the integral \begin{equation}\label{eq7766LN.220}
		\int_1^x \frac{B(z)}{z^2}dz=	\int_2^x\frac{1}{z(\log \log z)^{1/2-\varepsilon}} dz\gg \frac{\log x}{(\log \log x)^{1/2-\varepsilon}},
	\end{equation}
	for sufficiently large $x\geq1$, the assumption is false. Hence, it implies that there is an absolute upper bound $B(x)\ll x(\log \log x)^{-1/2+\varepsilon}$.
\end{proof}
\subsection{Single-Sign Patterns Liouville Characteristic Functions} \label{S8833LN}
The analysis of single-sign pattern characteristic function 
\begin{equation}\label{eq8833LN.100}
	\lambda^{\pm}(n)=\left( \frac{1\pm\lambda(n)}{2}\right)= 	\begin{cases}
		1 &\text{ if } \lambda(n)=\pm1,\\
		0 &\text{ if } \lambda(n)\ne\pm1,
	\end{cases}
\end{equation}
of the subset of integers
\begin{equation}\label{eq8833LN.110}
	\mathcal{N}_{\lambda}^{\pm}	=	\{n\geq 1: \mu(n)=\pm\}
\end{equation} 
is well known. Here, the same idea is extended to the shifted primes.

\begin{lem}\label{lem8833LN.200A} Let $a\ne0$ be an integer, and let $\lambda(n)\in \{-1,1\}$ be the Liouville function. Then,
	\begin{eqnarray}\label{eq8833LN.200A}
		\lambda^{\pm}(a,n)&=&\left( \frac{1\pm\lambda(n+a)}{2}\right)\\	&=&
		\begin{cases}
			1 &\text{ if } \lambda(n+a)=\pm1,\\
			0 &\text{ if } \lambda(n+a)\ne\pm1,\nonumber\\
		\end{cases}
	\end{eqnarray}
	are the characteristic functions of the subset of primes
	\begin{equation}\label{eq8833LN.210A}
		\mathcal{N}_{\lambda}^{\pm }(a)	=	\{n\geq 1: \lambda(n+a)=\pm1\}.
	\end{equation} 
\end{lem}

\subsection{Double-Sign Patterns Liouville Characteristic Functions} \label{S8855LN}
The analysis of single-sign pattern characteristic functions is extended here to the double-sign patterns
\begin{equation}\label{eq8855LN.100}
	\lambda(n)=\pm1 \quad \text{ and }\quad 	\lambda(n+t)=\pm1, 
\end{equation}
where $t\ne0$ is a small integer, and $n\geq1$ is an integer. Some of the research appears in \cite{HA1986}, \cite[Corollary 1.7]{TT2015}, \cite{KP1986}, \cite{SA2022}, and similar literature.

\begin{lem}\label{lem8855LN.200} Let $t\in \Z$ such that $t\ne0$ be small fixed integer, and let $\lambda(n)\in \{-1,1\}$ be the Liouville function. Then,
	\begin{eqnarray}\label{eq8855LN.200}
		\lambda^{\pm \pm}(t,p)&=&\left( \frac{1\pm \lambda(n)}{2}\right)\left( \frac{1\pm\lambda(n+t)}{2}\right)\\	&=&
		\begin{cases}
			1 &\text{ if } \lambda(n)=\pm1,\mu(n+t)=\pm1,\\
			0 &\text{ if } \lambda(n)\ne\pm1,\mu(n+t)\ne\pm1,\nonumber\\
		\end{cases}
	\end{eqnarray}
	are the characteristic functions of the subset of integers
	\begin{equation}\label{eq8855LN.210}
		\mathcal{B}^{\pm \pm}(t)	=	\{n\geq 1: \lambda(n)=\pm1, \lambda(n+t)=\pm1\}.
	\end{equation} 
\end{lem}

\subsection{Single-Sign Pattern Liouville Counting Functions} \label{S8844LN}
The single-sign patterns $\lambda(n)=1$ and $\lambda(n)=-1$ single-sign pattern counting functions over the integers have the forms
\begin{equation}\label{eq8844LN.100A}
	Q^{+}(x)=\sum_{\substack{n\leq x\\ \lambda(n)=1}}1=\sum_{n\leq x}\left( \frac{1+\lambda(n)}{2}\right) =\frac{1}{2}[x]+O\left( xe^{-c\sqrt{\log x}}\right), 
\end{equation}
and 
\begin{equation}\label{eq8844LN.100B}
	Q^{-}(x)=\sum_{\substack{n\leq x\\ \lambda(n)=-1}}1=\sum_{n\leq x}\left( \frac{1-\lambda(n)}{2}\right) =\frac{1}{2}[x]+O\left(xe^{-c\sqrt{\log x}}\right), 
\end{equation}
where $[x]$ is the largest integer function, respectively. In terms of the single-sign pattern counting functions, the summatory function has the asymptotic formula
\begin{equation}\label{eq8844LN.100C}
	Q(x)=\sum_{n\leq x}\lambda(n)=Q^{+}(x)-Q^{-}(x)=O\left( xe^{-c\sqrt{\log x}}\right). 
\end{equation}
Basically, it is a different form of the Prime Number Theorem
\begin{equation}\label{eq8844LN.100D}
	\pi(x)=\li(x)+O\left(xe^{-c\sqrt{\log x}} \right) ,	
\end{equation}
where $\li(x)=\int_2^2(\log t)^{-1}dt$ is the logarithm integral, and $c>0$ is an absolute constant, see \cite[Eq.~27.12.5]{DLMF}, \cite[Theorem 3.10]{EL1985}, et alii. \\

\subsection{Double-Sign Patterns in Liouville Counting Functions} \label{S8877LN}
The counting functions for the single-sign patterns $\lambda(n)=1$ and $\lambda(n+t)=-1$ are extended to the counting functions for the double-sign patterns 
\begin{equation}\label{eq8877LN.100}
	(\lambda(n),\lambda(n+t))=(\pm1,\pm1).
\end{equation}

The double-sign patterns counting functions are defined by
\begin{equation}\label{eq8877LN.110A}
	Q^{++}(t,x)=\sum_{\substack{n\leq x\\ \lambda(n)=1,\; \lambda(n+t)=1}}1=\sum_{\substack{n\leq x\\ n\in \mathcal{B}^{++}(t)}}1,
\end{equation}
\begin{equation}\label{eq8877LN.110B}
	Q^{+-}(t,x)=\sum_{\substack{n\leq x\\ \lambda(n)=1,\; \lambda(n+t)=-1}}1=\sum_{\substack{n\leq x\\ n\in \mathcal{B}^{+-}(t)}}1,
\end{equation}
\begin{equation}\label{eq8877LN.110C}
	Q^{-+}(t,x)=\sum_{\substack{n\leq x\\ \lambda(n)=-1,\; \lambda(n+t)=1}}1=\sum_{\substack{n\leq x\\ p\in \mathcal{N}_{\lambda}^{-+}(t)}}1,
\end{equation}
\begin{equation}\label{eq8877LN.110D}
	Q^{- -}(t,x)=\sum_{\substack{n\leq x\\ \lambda(n)=-1,\; \lambda(n+t)=-1}}1=\sum_{\substack{n\leq x\\ p\in \mathcal{N}_{\lambda}^{- -}(t)}}1.
\end{equation}

The double-sign patterns counting functions \eqref{eq8877LN.110A} to \eqref{eq8877LN.110D} are precisely the cardinalities of the subsets of integers
\begin{multicols}{2}
	\begin{enumerate}
		\item $\mathcal{N}_{\lambda}^{++}(t)\subset \N$ ,
		\item $\mathcal{N}_{\lambda}^{+-}(t)\subset \N$ ,
		\item $\mathcal{N}_{\lambda}^{-+}(t)\subset \N$ ,
		\item $\mathcal{N}_{\lambda}^{--}(t)\subset \N$ ,
	\end{enumerate}
\end{multicols}
defined in \eqref{eq8855LN.210}. 

\begin{lem}\label{lem8877LN.300} Let $x\geq1$ be a large number, and let $t\in \Z$ such that $t\ne0$, be a fixed integer. If $\lambda: \mathbb{Z} \longrightarrow \{-1,1\}$ is the Liouville function, then, 
	\begin{equation}\label{eq8877LN.300}
		Q^{\pm \pm}(t,x)= \frac{1}{4}[x]\pm \frac{1}{4}Q(t,x)+O\left(xe^{-c\sqrt{\log x}} \right) ,\nonumber
	\end{equation}
	where $c>0$ is an absolute constant.
\end{lem}

\begin{proof} Without loss in generality, consider the pattern $(\lambda(n),\lambda(n+t))=(+1,+1)$. Now, use Lemma \ref{lem8855LN.200} to express the double-sign pattern counting function as 
	\begin{eqnarray}   \label{eq8877LN.310}
		4Q^{++}(t,x)&=&\sum_{n\leq x}\lambda^{++}(t,n)\nonumber\\
		&=&\sum_{n\leq x}\left( 1+\lambda(n)\right) \left( 1+\lambda(n+t)\right)\nonumber\\
		&=&\sum_{n\leq x}\left( 1+\lambda(n)+\lambda(n+t)+\lambda(n)\lambda(n+t)\right)\\
		&=&\sum_{n\leq x}1+\sum_{n\leq x}\lambda(n)  +\sum_{n\leq x}\lambda(n+t)+\sum_{n\leq x}\lambda(n)\lambda(n+t)\nonumber\\
		&\geq&0\nonumber.
	\end{eqnarray}
	The first three finite sums on the last line have the following evaluations or estimates.
	\begin{enumerate}
		\item $ \displaystyle \sum_{p\leq x}1=[x], $ where $[x]$ is the largest integer function.\\
		
		\item $ \displaystyle \sum_{n\leq x}\lambda(n)=O \left (xe^{-c\sqrt{\log x}}\right )$, see Theorem \ref{thm2299.500}.\\
		
		\item $ \displaystyle \sum_{n\leq x}\lambda(n+t)=O \left (xe^{-c\sqrt{\log x}}\right )$, see Theorem \ref{thm2299.500}.
	\end{enumerate}
	In all cases $c>0$ is an absolute constant, unconditionally. Summing these evaluations or estimates verifies the claim for $Q^{++}(t,x)\geq0$. The verifications for the next three double-sign pattern counting functions $Q^{+-}(t,x)\geq0$, $Q^{-+}(t,x)\geq0$, and $Q^{--}(t,x)\geq0$ are similar.
\end{proof}

\subsection{Equidistribution of Double-Sign Patterns}\label{S5225LN}
The nontrivial result for the summatory Liouville function
\begin{eqnarray}\label{eq5225LN.200}
	\sum_{n\leq x}\lambda(n)
	&=&Q^{+}(x)-Q^{-}(x)=O\left( xe^{-c\sqrt{\log x}}\right)
\end{eqnarray}
has no main term. It vanishes because the number of single-sign patterns $Q^{+}(x)=\#\{n\leq x: \lambda(n)=1\}$ and $Q^{-}(x)=\#\{n\leq x: \lambda(n)=-1\}$ have the same cardinality. This implies that the single-sign patterns are equidistributed on the interval $[1,x]$, and each has the natural density $\delta_{\lambda}^{\pm}=1/2$, see \eqref{eq8844LN.100A} to \eqref{eq8877LN.110D} for more detail. This idea is extended in the proof of the equidistribution of the double-sign patterns. \\

Recall that $Q^{\pm\pm}(t,x)=\#\{n\leq x:\lambda(n)=\pm1 ,\lambda(n+t)=\pm1\}$, and the natural density of a double-sign pattern is defined by
\begin{equation}\label{eq5225LN.500}
	\delta_{\lambda}^{\pm\pm}(t)=  \lim_{x\to \infty}\;	\frac{\#\{n\leq x:\lambda(n)=\pm1 ,\lambda(n+t)=\pm1\}}{x}.
\end{equation}  

\begin{thm}\label{thm5225LN.500} Let $x\geq1$ be a large number, and let $t\ne0$ be a fixed integer. Then, the double-sign patterns $++$, $+-$, $-+$, and $--$ of the Liouville pair $\lambda(n), \lambda(n+t)$ are equidistributed on the interval $[1,x]$. In particular, each double-sign pattern has the natural density 
	\begin{equation}\label{eq5225LN.510}
		\delta_{\lambda}^{\pm\pm}(t)= \frac{1}{4}.\nonumber
	\end{equation}
\end{thm}
\begin{proof}Observe that for large $x\geq1$, $Q^{\pm\pm}(t,x)>0$, this follows from Lemma \ref{lem8877LN.300}. Next, compute the limit of the proportion of double-sign pattern
	\begin{eqnarray}\label{eq5225LN.550}
		\delta_{\lambda}^{\pm\pm}(t)&=&  \lim_{x\to \infty}\;	\frac{\#\{n\leq x:\lambda(n)=\pm1 ,\lambda(n+t)=\pm1\}}{x}\nonumber\\
		&=&\lim_{x\to \infty}\;	\frac{[x]+Q(t,x)+O\left( xe^{-c\sqrt{\log x}}\right)}{4x}\nonumber\\
		&=&\lim_{x\to \infty}\;	\frac{[x]+o(x)+O\left( xe^{-c\sqrt{\log x}}\right)}{4x}\nonumber\\
		&=&\frac{1}{4}\nonumber,
	\end{eqnarray} 
	since $Q(t,x)\ll  x(\log \log x)^{-1/2+\varepsilon}$, see Lemma \ref{lem7766LN.200}. This proves that the double-sign patterns $++$, $+-$, $-+$, and $--$ are equidistributed on the interval $[1,x]$
	as $x\to \infty$.
\end{proof}

\begin{exa}\label{exa5225LN.600}{\normalfont Let $t=1$. By Theorem \ref{thm5225LN.500}, in any sufficiently large interval $[1,x]$, the number of any double-sign pattern $\lambda(n)=\pm1, \lambda(n+1)=\pm1$ is
		\begin{equation}\label{eq5225LN.610}
			Q^{\pm\pm}(t,x)	=\delta^{\pm\pm}(t)x+o(x)=\frac{1}{4}x+o(x).
		\end{equation}		
		The numerical data for $x=10^4$, shows that the actual value of the autocorrelation function is
		\begin{equation}\label{eq5225LN.630}
			\sum_{n\leq x}\lambda(n)\lambda(n+1)
			=112,
		\end{equation}
		and the actual values of the double-sign pattern counting functions are tabulated below.
		\begin{center}
			\begin{tabular}{ c|c|c|c } 
				$\lambda(n)$ & $\lambda(n+1)$ &Actual Count &Expected $R^{\pm\pm}(1,x)$\\
				\hline 
				$+1$ & $+1$ & $9924/4$&$10000/4+o(x)$ \\ 
				$+1$ & $-1$ & $9888/4$ &$10000/4+o(x)$\\ 
				$-1 $& $+1 $& $9888/4$& $10000/4+o(x)$\\
				$-1$ &$ -1 $& $10300/4$&$10000/4+o(x)$ \\		
			\end{tabular}
		\end{center}
		Given the small scale of this experiment, $x=10^4$, the actual data fits the prediction very well. The tiny differences among the actual values, (in third column), and the prediction by the double-sign pattern counting functions $Q^{\pm\pm}(1,x)$ seem to be properties of the races between the different subsets of integers $\mathcal{N}_{\lambda}^{\pm \pm}(t)$ attached to the double-sign patterns, see \eqref{eq8855N.210}. For an introduction to the literature in comparative number theory, prime number races, and similar topics, see \cite{GM2004}, et cetera.
	}
\end{exa}

\begin{exa}\label{exa5225N.700}{\normalfont Let $t=1$. This is a numerical example over a short interval $[x,x+y]=[10^7,10^7+10^3]$. There is no theoretical result for this case. However, the number of any double-sign pattern $\lambda(n)=\pm1, \lambda(n+1)=\pm1$ should be
		\begin{equation}\label{eq5225N.710}
			Q^{\pm\pm}(t,x)	=\delta^{\pm\pm}(t)x+o(x)=\frac{1}{4}x+o(x).
		\end{equation}		
		The numerical data shows that the actual value of the autocorrelation function is
		\begin{equation}\label{eq5225N.730}
			\sum_{x\leq n\leq x+y}\lambda(n)\lambda(n+1)
			=27,
		\end{equation}
		and the actual values of the double-sign pattern counting functions are tabulated below.
		\begin{center}
			\begin{tabular}{ c|c|c|c } 
				$\lambda(n)$ & $\lambda(n+1)$ &Actual Count &Expected $Q^{\pm\pm}(1,x)$\\
				\hline 
				$+1$ & $+1$ & $1100/4$&$1000/4+o(x)$ \\ 
				$+1$ & $-1$ & $976/4$ &$1000/4+o(x)$\\ 
				$-1 $& $+1 $& $972/4$& $1000/4+o(x)$\\
				$-1$ &$ -1 $& $956/4$&$1000/4+o(x)$ \\		
			\end{tabular}
		\end{center}
	}
\end{exa}

The above analysis and data suggest the followings conditional results
\begin{equation}\label{eq5225N.650}
	Q^{\pm\pm}(t,x)	=\delta_{\lambda}^{\pm\pm}(t)x+O(x^{1/2+\varepsilon})=\frac{1}{4}x+O(x^{1/2+\varepsilon}),
\end{equation}		
and
\begin{equation}\label{eq5225.660}
	\sum_{n\leq x}\lambda(n)\lambda(n+1)
	=O(x^{1/2+\varepsilon}),
\end{equation} 
where $\varepsilon>0$ is arbitrarily small.

\subsection{Equivalent Subsets of Sign Patterns Patterns}\label{S9282LN}
The join double-sign-patterns counting functions is defined by
\begin{equation}\label{eq9282LN.100A}
	Q^{+}(t,x)=\sum_{\substack{n\leq x\\ \lambda(n)=1,\; \lambda(n+t)=1}}1+\sum_{\substack{n\leq x\\ \lambda(n)=-1,\; \lambda(n+t)=-1}}1,
\end{equation}
and
\begin{equation}\label{eq9282LN.100B}
	Q^{-}(t,x)=\sum_{\substack{n\leq x\\ \lambda(n)=1,\; \lambda(n+t)=-1}}1+\sum_{\substack{n\leq x\\ \lambda(n)=-1,\; \lambda(n+t)=+1}}1.
\end{equation}
\begin{lem}\label{lem9282LN.500} Let $x\geq1$ be a large number, and let $t\in \Z$ such that $t\ne0$, be a fixed integer. If $\lambda: \mathbb{Z} \longrightarrow \{-1,1\}$ is the Liouville function, then, 
	\begin{equation}\label{eq9282LN.500}
		Q^{\pm }(t,x)= \frac{1}{2}[x]+O\left(xe^{-c\sqrt{\log x}} \right) ,\nonumber
	\end{equation}
	where $c>0$ is an absolute constant.
\end{lem}

\begin{proof}By Theorem \ref{thm5225LN.500}, each of the sign pattern $++$ and $--$ are equidistributed over the interval $[1,x]$ as $x\to\infty$, and each has the same natural density $\delta_{\lambda}^{++}(t)=\delta_{\lambda}^{--}(t)=1/4$. Thus, each double-sign pattern counting functions has an asymptotic formula of the form
	\begin{equation}\label{eq9282LN.510}
		Q^{\pm\pm}(t,x)=\sum_{\substack{n\leq x\\ \lambda(n)=\pm1,\; \lambda(n+t)=\pm1}}1	=\frac{1}{4}x+E^{\pm\pm}(x),
	\end{equation}	
	where $E^{\pm\pm}(x)=o(x)$ is an error term. \\
	
	Let $\displaystyle \pi:\mathcal{Z} \longleftrightarrow \mathcal{Z}$ be a permutation of the subset of integers $\mathcal{Z}=\{1,2,3,\ldots, [x]\}$, defined by $\pi(n)=m$. Each integer $n\leq x$ such that $\lambda(n)\lambda(n+t)=+1$ is mapped to an unique integer $m=\pi(n)\leq x$ such that $\lambda(m)=+1$. Therefore, the two counting functions are the same
	
	\begin{equation}\label{eq9282LN.530}
		\sum_{\substack{n\leq x\\ \lambda(n)=1,\; \lambda(n+t)=1}}1\;+\;\sum_{\substack{n\leq x\\ \lambda(n)=-1,\; \lambda(n+t)=-1}}1\;+\;E_0(x)=\sum_{\substack{m\leq x\\\lambda(m)=+1}}1,
	\end{equation}	 
	up to an error term $E_0(x)=o(x)$. Substituting the evaluations in \eqref{eq9282LN.510}, and \eqref{eq8844LN.100A},  for each counting function yields
	\begin{equation}\label{eq9282LN.540}
		\left( \frac{1}{4}[x]+E^{++}(x)\right) +\left( \frac{1}{4}[x]+E^{--}(x)\right) +E_0(x)=\frac{1}{2}[x]+O\left(xe^{-c\sqrt{\log x}} \right).
	\end{equation}
	Therefore, the total sum of error terms is bounded by
	\begin{equation}\label{eq9282LN.550}
		E(x)=E^{++}(x)+E^{--}(x)+E(x)=O\left(xe^{-c\sqrt{\log x}} \right),
	\end{equation}
	where $c>0$ is an absolute constant.
\end{proof}

\subsection{Result for the Liouville Autocorrelation Function} \label{S2626N}
The result for the Liouville autocorrelation function is unconditional, and has a standard error term.

\begin{proof}[\textbf{Proof}] ({\bfseries Theorem \ref{thm7979LN.200}}) Rewriting the autocorrelation function in terms of the double-sign patterns counting functions, and applying Lemma \ref{lem9282N.500} lead to the asymptotic formula
	\begin{eqnarray}\label{eq2626N.220A}
		Q(t,x)&=&\sum_{p\leq x}\lambda(n)\lambda(n+t)\nonumber\\
		&=&Q^{++}(t,x)+Q^{--}(t,x)-Q^{+-}(t,x)-Q^{-+}(t,x) \nonumber\\
		&=&Q^{+}(t,x)-Q^{-}(t,x) \nonumber\\
		&=&\frac{1}{2} [x]+O \left (xe^{-c\sqrt{\log x}}\right ) -\left( \frac{1}{2} [x]+O \left (xe^{-c\sqrt{\log x}}\right ) \right)  \\
		&=&O \left (xe^{-c\sqrt{\log x}}\right ) \nonumber,  
	\end{eqnarray} 
	where $t\ne0$ is a small integer, and $c>0$ is an absolute constant.
\end{proof}

\begin{exa}\label{exa5225N.705}{\normalfont Let $t=1$. This is a numerical example over a short interval $[x,x+y]=[10^7,10^7+10^3]$. There is no theoretical result for this case. However, the number of any double-sign pattern $\lambda(n)=\pm1, \lambda(n+1)=\pm1$ should be
		\begin{equation}\label{eq5225LN.710}
			Q^{\pm\pm}(t,x)	=\delta_{\lambda}^{\pm\pm}(t)x+o(x)=\frac{1}{4}x+o(x).
		\end{equation}		
		The numerical data shows that the actual value of the autocorrelation function is
		\begin{equation}\label{eq5225LN.730}
			\sum_{x\leq n\leq x+y}\lambda(n)\lambda(n+1)
			=27,
		\end{equation}
		and the actual values of the double-sign pattern counting functions are tabulated below.
		\begin{center}
			\begin{tabular}{ c|c|c|c } 
				$\lambda(n)$ & $\lambda(n+1)$ &Actual Count &Expected $Q^{\pm\pm}(1,x)$\\
				\hline 
				$+1$ & $+1$ & $1100/4$&$1000/4+o(x)$ \\ 
				$+1$ & $-1$ & $976/4$ &$1000/4+o(x)$\\ 
				$-1 $& $+1 $& $972/4$& $1000/4+o(x)$\\
				$-1$ &$ -1 $& $956/4$&$1000/4+o(x)$ \\		
			\end{tabular}
		\end{center}
	}
\end{exa}

\section{Correlation Functions Problems} \label{S6666P}
\subsection{Mobius autocorrelation function problems}
\begin{prob}\label{P6666.200} {\normalfont Let $x\geq x_0$ be a large number, and let $a<b$ be small integers. Determine the oscillation results for the autocorrelation functions 	$$\sum_{ n\leq x}\mu(n+a)\mu(n+b)=\Omega(x^{\beta})$$ over the integers, and 
		$$\sum_{ p\leq x}\mu(p+a)\mu(p+b)=\Omega(x^{\beta})$$ over the shifted primes, where  $\beta \in (1/2,1)$. In \cite{KS2022}, there is some evidence and material for random autocorrelation functions.
	}
\end{prob}
\begin{prob}\label{P6666.205} {\normalfont Let $x\geq x_0$ be a large number, and let $a<b$ be small integers. Assume the RH. Determine the conditional results for the autocorrelation functions $$\sum_{ p\leq x}\mu(p+a)\mu(p+b)=O(x^{\beta})$$ over the integers, and  
		$$\sum_{ p\leq x}\mu(p+a)\mu(p+b)=O(x^{\beta})$$ over the shifted primes, where  $\beta \in (1/2,1)$.
	}
\end{prob}

\begin{prob}\label{P6666.215} {\normalfont Let $x\geq x_0$ be a large number, let $y\geq x^{1/2}$, and let $a<b$ be small integers. Determine a nontrivial upper bound for the autocorrelation function 
		$$\sum_{x\leq n\leq x+y}\mu(n+a)\mu(n+b)$$ over the integers, and $$\sum_{x\leq p\leq x+y}\mu(p+a)\mu(p+b)$$ over the shifted primes in the short interval $[x,x+y]$.
	}
\end{prob}

\begin{prob}\label{P6666.220} {\normalfont Let $x\geq x_0$ be a large number, and let $a_0<a_1<a_2$ be small integers. Either conditionally or unconditionally, use sign patterns techniques to verify the asymptotic for the autocorrelation function $$\sum_{ n\leq x}\mu(n+a_0)\mu(n+a_1)\mu(n+a_2)=O\left( \frac{x}{(\log x)^{c}}\right) $$ over the integers. 
	}
\end{prob}

\begin{prob}\label{P6666.225} {\normalfont Let $a_0<a_1<\cdots<a_{k-1}$ be a subset of small integers. Determine the maximal value of $k\geq2$ for which the autocorrelations
		$$\sum_{n\leq x}\mu(n+a_0)\mu(n+a_1)\cdots\mu(n+a_{k-1})$$ over the integers, and 
		$$\sum_{p\leq x}\mu(p+a_0)\mu(p+a_1)\cdots\mu(p+a_{k-1})$$ over the shifted primes,	have nontrivial upper bounds.
	}
\end{prob} 
\begin{prob}\label{P6666.230} {\normalfont Let $x\geq x_0$ be a large number, let $1\leq a<q\ll (\log x)^c$, with $c\geq0$, and let $a_0<a_1$ be small integers. Determine a nontrivial upper bound for the autocorrelation function 
		$$\sum_{\substack{x\leq n\\n\equiv a \bmod q}}\mu(n+a_0)\mu(n+a_1)$$ over the integers in arithmetic progression, and $$\sum_{\substack{x\leq n\\n\equiv a \bmod q}}\mu(p+a_0)\mu(p+a_1)$$ over the shifted primes in arithmetic progression.
	}
\end{prob}

\subsection{Liouville autocorrelation function problems}
\begin{prob}\label{P6666.200B} {\normalfont Let $x\geq x_0$ be a large number, and let $a<b$ be small integers. Determine the oscillation results for the autocorrelation functions 	$$\sum_{ n\leq x}\lambda(n+a)\lambda(n+b)=\Omega(x^{\beta})$$ over the integers, and 
		$$\sum_{ p\leq x}\lambda(p+a)\lambda(p+b)=\Omega(x^{\beta})$$ over the shifted primes, where  $\beta \in (1/2,1)$. In \cite{KS2022}, there is some evidence and material for random autocorrelation functions.
	}
\end{prob}
\begin{prob}\label{P6666.205B} {\normalfont Let $x\geq x_0$ be a large number, and let $a<b$ be small integers. Assume the RH. Determine the conditional results for the autocorrelation functions $$\sum_{ n\leq x}\lambda(n+a)\lambda(n+b)=O(x^{\beta})$$ over the integers, and  
		$$\sum_{ p\leq x}\lambda(p+a)\lambda(p+b)=O(x^{\beta})$$ over the shifted primes, where  $\beta \in (1/2,1)$.
	}
\end{prob}

\begin{prob}\label{P6666.215B} {\normalfont Let $x\geq x_0$ be a large number, let $y\geq x^{1/2}$, and let $a<b$ be small integers. Determine a nontrivial upper bounds for the autocorrelation function 
		$$\sum_{x\leq n\leq x+y}\lambda(n+a)\lambda(n+b)$$ over the integers, and $$\sum_{x\leq p\leq x+y}\lambda(p+a)\lambda(p+b)$$ over the shifted primes in the short interval $[x,x+y]$.
	}
\end{prob}

\begin{prob}\label{P6666.220B} {\normalfont Let $x\geq x_0$ be a large number, and let $a_0<a_1<a_2$ be small integers. Use sign patterns techniques to verify the asymptotic for the autocorrelation function $$\sum_{ n\leq x}\lambda(n+a_0)\lambda(n+a_1)\lambda(n+a_2)=O\left( xe^{-c\sqrt{\log x}}\right) $$ over the integers
	}
\end{prob}

\begin{prob}\label{P6666.225B} {\normalfont Let $a_0<a_1<\cdots<a_{k-1}$ be a subset of small integers. Determine the maximal value of $k\geq2$ for which the autocorrelations
		$$\sum_{n\leq x}\lambda(n+a_0)\lambda(n+a_1)\cdots\lambda(n+a_{k-1})$$ over the integers, and 
		$$\sum_{p\leq x}\lambda(p+a_0)\lambda(p+a_1)\cdots\lambda(p+a_{k-1})$$ over the shifted primes,	have nontrivial upper bounds.
	}
\end{prob} 
\begin{prob}\label{P6666.230B} {\normalfont Let $x\geq x_0$ be a large number, let $1\leq a<q\ll (\log x)^c$, with $c\geq0$, and let $a_0<a_1$ be small integers. Determine a nontrivial upper bound for the autocorrelation function 
		$$\sum_{\substack{x\leq n\\n\equiv a \bmod q}}\lambda(n+a_0)\lambda(n+a_1)$$ over the integers in arithmetic progression, and $$\sum_{\substack{x\leq n\\n\equiv a \bmod q}}\lambda(p+a_0)\lambda(p+a_1)$$ over the shifted primes in arithmetic progression.
	}
\end{prob}

\subsection{Intercorrelation functions problems}

\begin{prob}\label{P6666.300} {\normalfont Let $t\ne0$ be a small integer. Use the identities $\mu(n)=\lambda(n)\mu^2(n)$ and $\lambda(n)=\sum_{d^2\mid n}\mu(n/d^2)$, etc., to show that
		$$\sum_{n\leq x}\mu(n)\mu(n+t)=\sum_{d^2\leq x}\sum_{e^2\leq x}\mu(d)\mu(e)\sum_{\substack{n\leq x\\d^2\mid n\\e^2\mid  n+t}}\lambda(n)\lambda(n+t),$$
		and $$  \sum_{n\leq x}\lambda(n)\lambda(n+t)=\sum_{d^2\leq x}\sum_{e^2\leq x}\mu(d)\mu(e)\sum_{\substack{n\leq x\\d^2\mid n\\e^2\mid  n+t}}\mu(n/d^2)\mu((n+t)/e^2).$$
	}
\end{prob} 
\begin{prob}\label{P6666.310} {\normalfont Let $a<b$ be a pair of small integers. Use the identities $\mu(n)=\lambda(n)\mu^2(n)$ and $\lambda(n)=\sum_{d^2\mid n}\mu(n/d^2)$, etc., to find an easy way to prove that
		$$\sum_{n\leq x}\mu(n+a)\mu(n+b)=O\left( \frac{x}{(\log x)^c}\right) $$
		implies that $$  \sum_{n\leq x}\lambda(n+a)\lambda(n+b)=O\left( \frac{x}{(\log x)^c}\right),$$
		where $c>0$. Is the converse true?
	}
\end{prob}

\subsection{Hardy-Littlewood-Chowla conjecture problems}
\begin{prob}\label{P6666.400A} {\normalfont Let $a<b$ be a pair of small integers. Use sign patterns techniques to estimate the twisted finite sum
		$$\sum_{n\leq x}\Lambda(n)\mu(n+a)\mu(n+b) .$$
	}
\end{prob} 
\begin{prob}\label{P6666.400B} {\normalfont Let $a<b$ be a pair of small integers. Use sign patterns techniques to estimate the twisted finite sum
		$$\sum_{n\leq x}\Lambda(n)\lambda(n+a)\lambda(n+b).$$
	}
\end{prob} 


\currfilename.\\
\end{document}